\documentclass[11pt]{article}

by-\headheight \advance \topmargin by-\headsep 
\usepackage{hyperref}

\usepackage{todonotes} 

\usepackage[square,sort&compress,comma,numbers]{natbib}

\newcommand{\beq}{\begin{equation}}
\newcommand{\eeq}{\end{equation}}

\def\lloc{\mathop{\rm loc}\nolimits}
\def\D{{\nabla}}

   \def\cC{{\cal C}}

   \def\cM{{\cal M}}
   \def\cN{{\cal N}}

   \def\cS{{\cal S}}

   \def\Id{\mathop{\rm Id}\nolimits}

   \def\supp{\mathop{\rm supp}\nolimits}

\usepackage{graphicx}
\usepackage{amsmath,amssymb,amsthm,amsfonts,mathrsfs}
\usepackage{amssymb}
\usepackage{color}
\usepackage{colortbl}
\newtheorem{thm}{Theorem}[section]
\newtheorem{cor}[thm]{Corollary}
\newtheorem{lem}[thm]{Lemma}
\newtheorem{prop}[thm]{Proposition}

\theoremstyle{definition}

\theoremstyle{remark}
\newtheorem{rem}[thm]{Remark}

\numberwithin{equation}{section}

\DeclareMathSymbol{\C}{\mathalpha}{AMSb}{"43}

\newcommand{\N}{{\mathbb N}}
\newcommand{\R}{{\mathbb{R}}}

\def\ni{\noindent}

\newcommand{\bsub}{\begin{subequations}}
\newcommand{\esub}{\end{subequations}$\!$}

\begin{document}
\title{Standing waves for two-component elliptic system with critical growth in $\R^{4}$: the attractive case }
\author{Lun Guo\thanks{
School of Mathematics and Statistics, South-Central University for Nationalities, Wuhan 430074,  China; School of Mathematical Sciences, Fudan University, 200433,  Shanghai,  China.
 Email: \texttt{lguo@mails.ccnu.edu.cn}.
 }
\qquad
Qi Li\thanks{College of Science \& Hubei Province Key Laboratory of Systems Science in Metallurgical Process, Wuhan University of Science and Technology,  430065,  Wuhan, China.
  Email: \texttt{qili@mails.ccnu.edu.cn}.
  }
\qquad
Xiao Luo\thanks{School of Mathematics, Hefei University of Technology, 230009, Hefei, China.
Email: \texttt{luoxiao@hfut.edu.cn}.
}
\qquad
Riccardo Molle\thanks{Dipartimento di Matematica, Universit\`{a} di Roma ``Tor Vergata'', Via della Ricerca Scientifica n. 1, 00133,  Roma, Italy.
  Email: \texttt{molle@mat.uniroma2.it}.
  }}
\maketitle

\begin{abstract}
In this paper, we consider the following  two-component elliptic  system with critical growth
\begin{equation*}
  \begin{cases}
    -\Delta u+(V_1(x)+\lambda)u=\mu_1u^{3}+\beta uv^{2}, \ \  x\in \R^4, \\
    -\Delta v+(V_2(x)+\lambda)v=\mu_2v^{3}+\beta vu^{2}, \ \  x\in \R^4   , \\
  \end{cases}
\end{equation*}
where $V_j(x) \in L^{2}(\R^4)$ are nonnegative potentials and the nonlinear coefficients $\beta ,\mu_j$, $j=1,2$, are positive.
Here we also assume $\lambda>0$.
By variational methods combined with degree theory, we prove some results about the existence and multiplicity of positive  solutions under the hypothesis $\beta>\max\{\mu_1,\mu_2\}$. These results generalize the results for semilinear Schr\"{o}dinger equation on half space  by Cerami and Passaseo (SIAM J. Math. Anal., 28, 867-885, (1997)) to the above elliptic  system, while  extending  the existence result from  Liu and Liu (Calc. Var. Partial Differential Equations, 59:145, (2020)).

\end{abstract}

\vskip 0.05truein

\noindent {\it Keywords:} Elliptic system,  Critical growth, Positive solution, Lack of compactness. \\
{\bf 2020 Mathematics Subject Classification }: 35B33, 35J47, 35J50.
\vskip 0.1truein

\section{Introduction and main results}

\quad Our goal in this paper is to investigate the existence and multiplicity of positive  solutions, when the ground state solutions cannot be attained, for  the following two-component elliptic system
\begin{equation}\label{S}
  \begin{cases}
    -\Delta u+(V_1(x)+\lambda)u=\mu_1u^{3}+\beta uv^{2}, \ \ x\in \R^4, \\
    -\Delta v+(V_2(x)+\lambda)v=\mu_2v^{3}+\beta vu^{2}, \ \ x\in \R^4, \\
  \end{cases}
\end{equation}
where $V_1(x)$ and $V_2(x)$ are nonnegative potential and $\lambda, \mu_{j}, \beta   >0$,  $j=1,2$.
The motivation for studying  system \eqref{S} comes from the search for standing wave solutions
\beq
\label{ans}
(\Phi_1(x,t),\Phi_2(x,t))=(e^{i\lambda t}u(x), e^{i\lambda t}v(x))
\eeq
of the time-dependent coupled nonlinear Schr\"{o}diger system
\begin{equation}\label{S-2}
\begin{cases}
  -i \frac{\partial}{\partial t}\Phi_1=\Delta\Phi_{1}+\mu_1|\Phi_1|^{2}\Phi_1+\beta|\Phi_2|^{2}\Phi_1,\,\, x\in \R^N, \,\, t>0, \\
  -i \frac{\partial}{\partial t}\Phi_2=\Delta\Phi_{2}+\mu_2|\Phi_2|^{2}\Phi_2+\beta|\Phi_1|^{2}\Phi_2, \,\, x\in \R^N, \,\, t>0, \\
  \Phi_{j}=\Phi_{j}(x,t) \in \C, \ \ j=1,2,   \\
  \Phi_{j}(x,t)\to 0, \ \ \text{as}\, |x|\to +\infty, \,\, t>0, \,\, j=1,2,
\end{cases}
\end{equation}
where $i$ is the imaginary unit. The system \eqref{S-2} defined in $\R^2$ or $\R^3$  appears in many physical problems, as  the Hartree-Fock theory for a double condensate, that is, a binary mixture of Bose-Einstein condensates in two different hyperfine states $|1\rangle$ and $| 2\rangle$, see \cite{EGBB,T}. Physically, $\Phi_1$ and $\Phi_2$ are the corresponding condensate amplitudes,  $\mu_1$ and $\mu_2$, and $\beta$ are the intraspecies and interspecies scattering lengths respectively. The sign of $\beta$ determines whether the interactions of states $|1\rangle$ and $| 2\rangle$ are attractive ($\beta>0$) or repulsive $(\beta<0)$.   In the attractive case, the two components  tend to go along with each other, leading to synchronization. While  in the repulsive case,  the two components  tend to separate in different regions as  $\beta$ tends to negative infinity, leading to phase separation. System \eqref{S-2} also arises in nonlinear optics (see \cite{AA}). Due to the important application in physics, systems \eqref{S} and \eqref{S-2} in low dimensions  (dimension  $ N =1,2,3 $) has been studied extensively in the last decades.  By variational methods, Lyapunov-Schmidt reduction methods or bifurcation methods, the existence, multiplicity and properties of weak solutions for system \eqref{S} have been established in the literature under various assumptions. Since it seems almost impossible for us to give a complete list of references, we just refer the reader to \cite{AC1,AC2,BD,BWW,CZ1,CV,DL,DWW,LW1,LW2,LW3,LW4,MPi,MPS,NTTV,PW,S1,WW1,WW2,WW3} and the references therein, for related problems.
Note that, if $N \leq 3$, the nonlinearity and  coupling terms in system \eqref{S} are of subcritical growth with respect to Sobolev critical exponent.

Recently, system \eqref{S} defined in $\R^N$ with $N\geq 4$ has begun to catch  mathematicians' attention, see \cite{CP1,CLZ,CZ2,CZ3,LL,PPW,PPW2,PS,PST,PT,WZ} and the  references therein.
In these cases, the cubic nonlinearities and coupled terms are all of critical growth for $N=4$ and even super-critical growth  for $N\geq 5$, with respect to Sobolev critical exponent.
Thus, the study of the cases in high dimensions is much more complicated due to the lack of compactness.

In \cite{CZ2}, Chen and Zou considered a system of the type \eqref{S} restricted on smooth bounded domains $\Omega\subset\R^4$:
\begin{equation}\label{S-3}
  \begin{cases}
   -\Delta u+\lambda_1u=\mu_1u^{3}+\beta uv^{2}, \ \  x\in \Omega, \\
    -\Delta v+\lambda_2v=\mu_2v^{3}+\beta vu^{2}, \ \  x\in \Omega, \\
    u=v=0, \ \  x\in \partial\Omega.
  \end{cases}
\end{equation}
By using the Ekeland variational principle and the Mountain Pass theorem, they showed the existence of positive ground state solutions in some   ranges of $\lambda_1$, $\lambda_2$ and $\beta$.
For more general powers, Chen and Zou \cite{CZ3} showed that the dimension has a great influence on the existence of positive ground state solutions.
After that, the existence of sign-changing solutions to system \eqref{S-3} has been derived by Chen and Zou \cite{CLZ}, Peng et al. \cite{PPW2}.
In particular, for  $\lambda_1=\lambda_2=0$ in \eqref{S-3}, Peng et al. \cite{PPW} investigated a Coron-type problem and obtained a positive solution when $\Omega$ is bounded and has nontrivial topology.
For $m$-coupled equations with critical growth in a bounded domain with one or multiple small holes, the researchers \cite{PS,PST,PT} proved some existence and concentration results via Lyapunov-Schmidt reduction argument.
In contrast, there are only very few results for system \eqref{S} on the whole $\R^4$.
In \cite{WZ}, Wu and Zou  first proved the existence of positive ground state solutions for system \eqref{S} with  steep potential wells, and then showed the phenomenon of phase separation of ground state
solutions to system \eqref{S} as $\beta\to -\infty$.

Very recently,  the positive solutions  are investigated at higher energy levels than ground state energy level, see  \cite{CP1,LL} for example.  By considering the functional constrained on a subset of the Nehari manifold consisting of functions invariant with respect to a subgroup of  $O(N+1)$, Clapp and Pistoia \cite{CP1} proved that system \eqref{S} has infinitely many fully nontrivial solutions, which are not conformally equivalent.
In  \cite{LL}, by using a global compactness result together with Linking theorem,  Liu and Liu  extended the celebrated work \cite{BC} by Benci and Cerami on semilinear Schr\"{o}dinger equation to system \eqref{S} with $\lambda=0$.
More precisely, they proved that system \eqref{S} has a positive solution with its functional energy lying in $(c_{\infty}, \min\{\frac{\mathcal{S}^{2}}{4\mu_1}, \frac{\mathcal{S}^{2}}{4\mu_2}, 2c_{\infty} \})$,
where $c_{\infty}$, $\mathcal{S}$ denotes the ground state energy level and  the best constant in the  Sobolev embedding $D^{1,2}(\R^4)\hookrightarrow L^{2}(\R^4)$ respectively, whenever $V_1(x)$ and $V_2(x)$ satisfy the following assumptions:
\\
$(A_1)$ \, $V_{1}(x),V_{2}(x)\geq 0$, $\forall x\in \R^4$, \ $V_1(x)+V_2(x)\not\equiv 0$;
\\
$(A_2)$ \, $V_1(x), V_{2}(x)\in L^{2}(\R^4)$;\\
$(A_3)$ \, $V_1(x)$, $V_{2}(x)$ verify
\begin{align*} \label{V}
  &\frac{\beta-\mu_2}{2\beta-\mu_1-\mu_2}\|V_1\|_{L^2}+\frac{\beta-\mu_1}{2\beta-\mu_1-\mu_2}\|V_2\|_{L^2} \nonumber \\
  & < \min \left\{\sqrt{\frac{\beta^2-\mu_1\mu_2}{\mu_1(2\beta-\mu_1-\mu_2)}},\,  \sqrt{\frac{\beta^2-\mu_1\mu_2}{\mu_2(2\beta-\mu_1-\mu_2)}},\,  \sqrt{2} \right\} \mathcal{S}-\mathcal{S}.
\end{align*}

\noindent Note that $V_j\in L^2(\R^4)$ and $\lambda=0$  imply that the only limit problem for \eqref{S} is
$$
  \begin{cases}
    -\Delta u=\mu_1 u^{3}+\beta uv^{2},   \\
    -\Delta v=\mu_2 v^{3}+\beta vu^{2},  \\
    u,v\in D^{1,2}(\R^4)
  \end{cases}
$$
 and this fact was a key ingredient in the approach of \cite{LL}.

 On the other hand, $\lambda>0$ is very significant by its physical meaning, as one can readily seen in the ansatz \eqref{ans}.
The purpose of this work is to investigate the existence and multiplicity of positive solutions for system \eqref{S}
in this case.

Before stating our results, we first introduce some definitions.
For any $\beta\neq 0$, the system \eqref{S} possesses a trivial solution $(0,0)$ and can have semi-trivial solutions,
that are solutions of the form $(u,0)$ or $(0,v)$.
We look for solutions of system \eqref{S} which are different from the preceding ones. We say a solution $(u,v)$ nontrivial if both $u\neq 0$ and  $v\neq 0$.   Moreover, we call a nontrivial solution  non-negative if both $u\ge 0$, $u\not\equiv 0$, and $v\ge 0$, $v\not\equiv 0$.
The main results of  our paper are stated as follows.

%


\begin{thm}\label{Th1.2}
Let $\lambda  >0$, $\beta>\max\{\mu_1,\mu_2\}$.
Assume that $V_1(x),V_2(x)$ satisfy   $(A_1)$ and $(A_2)$,
then there exists $\lambda^*>0$ such that  system \eqref{S} has at least a nontrivial non-negative solution for every $\lambda \in (0,\lambda^*)$.
If $V_1(x),V_{2}(x)$ satisfy $(A_1)$-$(A_3)$,
then there exists $\lambda^{**}$, with $\lambda^{**}\le\lambda^*$, such that   system \eqref{S} has at least two distinct nontrivial non-negative solutions for every $\lambda \in (0, \lambda^{**})$.
\end{thm}

If the potentials have some more regularity, then the   nontrivial non-negative  solutions provided by Theorem \ref{Th1.2} are actually  positive.
According to this question, in the following proposition we collect some known regularity results.
\begin{prop}
\label{Preg}
Let $(u,v)$ be a  solution of system \eqref{S}.
\begin{itemize}
\item[$a)$]
If  $V_{j}\in L^{q_j}_{\lloc}(\R^4)$ for $q_j>2$, $j=1,2$, then   $u,v\in\cC^{0,\alpha}_{\lloc}(\R^4)$.
Moreover, if $(u,v)$ is a nontrivial  non-negative solution, then $ u(x)>0$, $v(x)>0$, for all $x\in\R^4$.

\item[$ b)$] If  $V_j\in L^{q_j}_{\lloc}(\R^4)$ for $q_j> 4$, $j=1,2$,  then $u,v\in\cC^{1,\alpha}_{\lloc}(\R^4)$.

\item[$ c)$] If $V_j\in\cC^{0,\alpha_j}_{\lloc}(\R^4)$, then $u,v\in \cC^2(\R^4)$ and $(u,v)$ is a classical solution of system \eqref{S}.

\end{itemize}

\end{prop}

We prove Theorem \ref{Th1.2} by variational methods.
The main difficulties are the loss of compactness and the need to distinguish the nontrivial solutions from the semi-trivial ones.
To front these problems in the case $\lambda=0$, Liu and Liu \cite{LL} first  obtain a global compactness result for system \eqref{S}, i.e. they give a complete description for the Palais-Smale sequences of the corresponding energy functional.
Adopting this description, they  succeed in proving the existence of a critical level if $\|V_{1}\|_{L^2}+\|V_{2}\|_{L^2}$ is suitable small.
However, the methods in \cite{LL} cannot work well in our paper.
First, we remark that the case $\lambda>0$ presents some more difficulties because in such a case the natural space to work in is $H^{1}(\R^4)\times H^{1}(\R^4)$, while typically the global compactness result works in $D^{1,2}(\R^4)\times D^{1,2}(\R^4)$.
Second, the classical Linking theorem used in \cite{LL} can only prove the existence of a critical level, while it fails to obtain the multiplicity of critical levels.
In this paper, we follow some ideas from \cite{AFM,CM,CP2} together with new techniques of analysis for Schr\"{o}dinger system, and prove the multiplicity of positive solutions of system \eqref{S} by degree theory.

\begin{rem}
We would like to point out that the results above still hold  even if  we replace the constant $\lambda$ in each equation of system \eqref{S} by different $\lambda_1$ and $\lambda_2$ respectively, where $\lambda_1,\lambda_2 \in [0, \infty)$.
Similar to the proof of Theorem \ref{Th1.2}, we only need to make a little adjustment of Theorem \ref{th3.1} below and restrict  $\max\{\lambda_1,\lambda_2 \} \in (0,\lambda^{*})$.
\end{rem}

\begin{rem}
Let $(u_h,v_{h})$, $(u_l,v_{l})$ be the solutions obtained in Theorem \ref{Th1.2}, such that the energy of solution $(u_h,v_{h})$ is higher than the energy  of $(u_{l},v_{l})$.
Then  by the energy estimates provided in the proof, one can understand that $(u_{h},v_{h})$ converges to the solution obtained in the case $\lambda=0$ (see \cite{LL}), while $(u_l,v_l)$ flattens and vanishes as $\lambda\to 0$.
\end{rem}

\begin{rem}
 The hypothesis
   $(A_1)$  in the first part of Theorem \ref{Th1.2} is necessary  due to the non-existence result in Lemma 2.6. In order to find another nontrivial solution with its energy  lying in energy interval $(c_{\infty}, \min\{ \frac{\mathcal{S}^2}{4\mu_1}, \frac{\mathcal{S}^2}{4\mu_1},2c_{\infty} \})$, we need an additional hypothesis $(A_3)$.
Here, the role of $(A_3)$ is to guarantee that the energy level of suitable Palais-Smale sequences we obtained is less than  $\min\{ \frac{\mathcal{S}^2}{4\mu_1}, \frac{\mathcal{S}^2}{4\mu_1},2c_{\infty} \}$ and so the solution we get is nontrivial.
\end{rem}


\begin{rem}
  We remark that in our proof of  Theorem \ref{Th1.2},  the result on  classification of positive solutions for  system \eqref{S2} below plays an important role in restoring the compactness of Palais-Smale sequence (see Corollary \ref{co3.4}).   To our best knowledge,  the classification of positive solutions for  system \eqref{S2} is still open for $\beta<\min\{\mu_1,\mu_2\}$.  It would be quite interesting to give an existence result in this case.
  \end{rem}

The paper is organized as follows. In section 2,  we introduce the variational framework and obtain a nonexistence result.
In section 3, we prove a global compactness result.
In section 4, we introduce some technics and give some estimates that will be used in Section 5, where we will prove Theorem \ref{Th1.2} and recall the tools and some references for the proof of Proposition \ref{Preg}.
Our notations are standard.
We will use $C$, $C_{i}$, $i\in\mathbb{N}$, to denote positive constants that can differ from line to line.

\vskip 0.1truein
\section{Variational framework}\label{s2}

\qquad  
 Let $X$ and $Y$ be two
 Banach spaces with norms $\|\cdot\|_{X}$ and $\|\cdot\|_{Y}$,   we define the standard norm on the product space $X\times Y$ as following
 $$
 \|(x,y)\|_{X\times Y}^{2}=\|x\|_{X}^{2}+\|y\|_{Y}^{2}.
 $$
 In particular, if $X$ and $Y$ are two Hilbert spaces with inner products $\langle\cdot,\cdot\rangle_{X}$ and
 $\langle\cdot,\cdot\rangle_{Y}$, then $ X\times Y$ is also a Hilbert space with the inner product
 $$
 \langle(x,y),(\varphi,\phi)\rangle_{X\times Y}=\langle x,\varphi\rangle_{X}+\langle y,\phi\rangle_{Y}.
 $$
 Let the Sobolev space $D^{1,2}(\R^4)$  be endowed with the usual inner product and  norm
$$
\langle u,v\rangle_{D^{1,2}}= \int_{\R^4}\nabla u \nabla v \, dx \ \ \text{and} \ \ \|u\|_{D^{1,2}}^{2}=\int_{\R^4}|\nabla u|^{2}dx.$$
Then, we define the product space $H_0:=D^{1,2}(\R^4)\times D^{1,2}(\R^4)$ equipped with the norm
$$
\|(u,v)\|_{H_0}^2:=\|(u,v)\|_{D^{1,2}\times D^{1,2}}^2=\int_{\R^4}\big(|\nabla u|^{2}+|\nabla v|^{2}\big)dx.
$$%
For $\lambda>0$,  we consider the Hilbert space $H^{1}(\R^4)$ endowed with the inner product
$$
\langle u,v\rangle_{H^{1}}= \int_{\R^4}(\nabla u \nabla v+\lambda uv) \,dx.
$$
The corresponding norm is given by
$$ \|u\|_{H^{1}}^{2}=\int_{\R^4}(|\nabla u|^{2}+ \lambda |u|^{2})\,dx.$$
Throughout   this  paper,  we  will work in the  Hilbert space $H_{\lambda}:=H^{1}(\R^4)\times H^{1}(\R^4)$,  with the  norm
\begin{equation*}
  \|(u,v)\|_{H_{\lambda}}^{2}:= \|(u,v)\|_{H^{1}\times H^{1}}^{2}=\int_{\R^4}\big(|\nabla u|^{2}+ \lambda |u|^{2}\big)dx+ \int_{\R^4}\big(|\nabla v|^{2}+ \lambda |v|^{2}\big)dx.
\end{equation*}

Since we are looking for non-negative solutions of system \eqref{S},  we shall to consider the modified system
\begin{equation}\label{S-4}
  \begin{cases}
    -\Delta u+(V_1(x)+\lambda)u=\mu_1(u^+)^{3}+\beta (u^+)(v^+)^{2}, \ \ x\in \R^4, \\
    -\Delta v+(V_2(x)+\lambda)v=\mu_2(v^+)^{3}+\beta (v^+)(u^+)^{2}, \ \ x\in \R^4,\\
  \end{cases}
\end{equation}
where, as usual, we denote  $u^{\pm}=\max\{0,\pm u\}$,  $v^{\pm}=\max\{0,\pm v\}$,  so that $u=u^+-u^-$, $v=v^+-v^-$.
\begin{rem}
\label{rem+}
If $(u,v)$ is a nontrivial solution of system \eqref{S-4}, then it is a non-negative solution of system \eqref{S}. Indeed, multiplying  in system \eqref{S-4}  the first equation by $u^-$ and the second equation by $v^-$, and integrating, we obtain
\begin{equation*}
  \int_{\R^4} \big(|\nabla (u^-)|^{2}+|\nabla (v^-)|^2\big)dx+\int_{\R^4}(V_1(x)+\lambda)(u^{-})^{2}dx
  +\int_{\R^4}(V_2(x)+\lambda)(v^{-})^{2}dx=0,
\end{equation*}
which implies that $\|(u^-,v^-)\|_{H_{\lambda}}^{2}=0$.
Thus, $(u,v)$ is a non-negative solution of system \eqref{S-4} and also of system \eqref{S}. In the sequel, we focus our attention  on proving the existence and multiplicity of  nontrivial solutions of system \eqref{S-4}.
\end{rem}

System \eqref{S-4} has a variational structure.
Its solutions correspond to the critical points of  the energy functional  $I:H_{\lambda} \to \R$ defined by
\begin{align*}
  I(u,v)&=\frac{1}{2}\int_{\R^4}\big(|\nabla u|^{2}+|\nabla v|^{2}\big)dx+\frac{1}{2}\int_{\R^4}\big(  \big(V_{1}(x)+\lambda )u^2+(V_2(x)+\lambda \big)v^2 \big)dx \\
  & \qquad -\frac{1}{4}\int_{\R^4}\big(\mu_1(u^+)^4+2\beta (u^+)^2(v^+)^2+\mu_2 (v^+)^4     \big)dx.
\end{align*}
Since $V_1(x), V_2(x)\in L^{2}(\R^4)$, then in view of H\"{o}lder inequality,  $I$ is well defined in $H_{\lambda}$ and belongs to  $\cC^{1}(H_{\lambda}, \R)$.
In what follows, we set the infimum
\begin{equation*}
  c=\inf_{(u,v)\in \mathcal{N}}I(u,v)
\end{equation*}
defined on the Nehari manifold
\begin{equation*}
 \mathcal{ N}:=\{(u,v) \in H_{\lambda} \ : \ (u,v)\neq (0,0), \langle I'(u,v),(u,v)\rangle=0\}.
\end{equation*}
Here
\begin{align*}
  \langle I'(u,v),(u,v) \rangle &=\int_{\R^4}\big(|\nabla u|^{2}+|\nabla v|^{2}\big)dx+\int_{\R^4}\Big(  \big(V_{1}(x)+\lambda)u^2+(V_2(x)+\lambda\big)v^2 \Big)dx \\
  & \qquad -\int_{\R^4}\big(\mu_1(u^+)^4+2\beta (u^+)^2(v^+)^2+\mu_2 (v^+)^4     \big)dx,
\end{align*}
 and
\beq
\label{IN}
\begin{split}
  I(u,v)&=\frac{1}{4}\int_{\R^4}\big(|\nabla u|^{2}+|\nabla v|^{2}\big)dx+\frac{1}{4}\int_{\R^4}\big(  \big(V_{1}(x)+\lambda )u^2+(V_2(x)+\lambda \big)v^2 \big)dx \\
  &= \frac{1}{4}\int_{\R^4}\big(\mu_1(u^+)^4+2\beta (u^+)^2(v^+)^2+\mu_2 (v^+)^4     \big)dx \qquad \forall(u,v)\in\cN.
\end{split}
\eeq

We would like to point out that the existence of nontrivial solutions to  system \eqref{S-4} depends heavily on the study of the corresponding  problem
\begin{equation}\label{S2}
  \begin{cases}
    -\Delta u=\mu_1(u^+)^{3}+\beta (u^+)(v^+)^{2}, \ \  x\in \R^4, \\
    -\Delta v=\mu_2(v^+)^{3}+\beta (v^+)(u^+)^{2}, \ \ x\in \R^4.
  \end{cases}
\end{equation}
In this case, the analogues of $I$, $\mathcal{N}$  and $c$ will be denoted by $I_{\infty}$,  $\mathcal{N}_{\infty}$ and $c_{\infty}$  respectively. More precisely, for each $(u,v)\in H_0$, we define
\begin{equation*}
  I_\infty(u,v)=\frac{1}{2}\int_{\R^4}\big(|\nabla u|^{2}+|\nabla v|^{2}\big)dx-\frac{1}{4}\int_{\R^4}\big(\mu_1(u^+)^4+2\beta (u^+)^2(v^+)^2+\mu_2 (v^+)^4     \big)dx,
\end{equation*}
and set
\begin{equation*}
  c_{\infty}=\inf_{(u,v)\in \mathcal{N}_{\infty}}I_{\infty}(u,v),
\end{equation*}
where
\begin{equation*}
 \mathcal{N}_{\infty}:=\{(u,v) \in H_0 \ : \ (u,v)\neq (0,0), \,\, \langle I'_{\infty}(u,v),(u,v)\rangle=0\}.
\end{equation*}

Let us recall the following non-degenerate property:
\begin{prop}
\label{PN}
The following estimate holds:
\beq
\label{EN}
\inf\{\|(u,v)\|_{H_0} \  : \ (u,v)\in\cN_\infty\}=\cC>0
\eeq
\end{prop}

\begin{proof}
Inequality \eqref{EN} follows from
$$
0=\langle I'_{\infty}(u,v),(u,v)\rangle\ge \|(u,v)\|_{H_0}^2-C\|(u,v)\|_{H_0}^4,\qquad\forall (u,v)\in\cN_\infty.
$$
\end{proof}

In order to avoid semi-trivial solutions in searching for nontrivial solutions of system \eqref{S-4},   we first  investigate the scalar equations
\begin{equation*}
   -\Delta u+(V_j(x)+\lambda )u=\mu_j(u^{+})^{3},  \ \ u\in H^{1}(\R^4), \ \  j \in \{1,2\},
\end{equation*}
whose associated energy functional  is
\begin{equation*}
  J_{j}(u)=\frac{1}{2}\int_{\R^4}|\nabla u|^{2}dx+\frac{1}{2}\int_{\R^4}(V_{j}(x)+\lambda)u^2dx-\frac{1}{4}\int_{\R^4}\mu_j(u^+)^4dx.
\end{equation*}
Let us  consider the infimum
\begin{equation*}
  m_{j}=\inf_{u\in \mathcal{M}_{j}}J_{j}(u),
\end{equation*}
where
\begin{equation*}
\mathcal{M}_{j}=\Big\{ u\in H^{1}(\R^4)\  : \ u\neq 0, \,\, \int_{\R^4}|\nabla u|^{2}dx+\int_{\R^4}(V_{j}(x)+\lambda )u^2dx=\int_{\R^4}\mu_{j}(u^+)^4 dx    \Big\}.
\end{equation*}

 Observe that for every $u\in H^1(\R^4)\setminus\{0\}$ there exists a unique $\bar t=\bar t(j,u)>0$ such that $\bar t\, u\in \cM_j$.
Moreover
\beq
\label{N}
J_j(\bar t\, u)=\max_{t\ge 0}J_j(t,u).
\eeq
We call $\bar t\, u$ the projection of $u$ on  $\cM_j$

For the case $V_j=\lambda =0$, we then denote the analogues of $J_j$, $m_j$, $\mathcal{M}_j$ by $J_{j}^{\infty}$, $m_{j}^{\infty}$, $\mathcal{M}_{j}^{\infty}$ respectively.

Next, we  introduce some  results on the infimum $m_{j}^{\infty}$, which will be used later.
By computation we can easily prove that $m_{j}^{\infty}$ is attained by $W_{\delta,y,j}(x)=\mu_{j}^{-\frac{1}{2}}U_{\delta,y}$, where the Aubin-Talenti function
\begin{equation}\label{e-2.2}
  U_{\delta,y}(x)=2^{ 3/2}\, \frac{ \delta}{\delta^2+|x-y|^{2}}, \ \ \delta>0, \ \ y\in \R^4
\end{equation}
is the unique positive solution of
\begin{equation}\label{e-2.3}
  -\Delta u=u^{3}, \ \ x\in \R^4
\end{equation}
(see \cite{Au,Ta}).
Moreover,
\begin{equation*}
  \int_{\R^4}|\nabla U_{\delta,y}|^2dx=\int_{\R^4}|U_{\delta,y}|^{4}dx=\mathcal{S}^{2},
\end{equation*}
where $\mathcal{S}$ is the sharp constant of the Sobolev embedding $D^{1,2}(\R^4)\hookrightarrow L^{4}(\R^4)$:
$$
\cS=\inf_{u\in D^{1,2}(\R^4)\setminus\{0\}}\frac{\int_{\R^4}|\D u|^2dx}{\left(\int_{\R^4}|u|^4dx\right)^{1/2}}.
$$
Then, a direct calculations show that
\beq
\label{mj}
m_{j}^{\infty}=\frac{1}{4\mu_{j}} \mathcal{S}^2.
\eeq

\begin{lem}\label{lm2.1}
Suppose that $\lambda> 0$ and  $V_{1}(x), V_2(x)$ satisfy the assumptions $(A_1)$-$(A_2)$, then $m_j=m_{j}^{\infty}=\frac{\mathcal{S}^{2}}{4\,\mu_j}$.
\end{lem}
\begin{proof}
Let $u\in \mathcal{M}_j$ be  chosen arbitrarily  and let $t_\infty>0$ be such that $t_\infty u\in \cM^\infty_j$.
As $\lambda>0$ and $V_{j}(x)\geq 0$ in $\R^4$, taking also into account \eqref{N} we obtain
$$
J_{j}(u)\geq J_j(t_\infty u)\ge J_j^\infty(t_\infty u)\ge m_{j}^{\infty}.
$$

So we have $m_j\geq m_{j}^{\infty}$.

In order to show that the opposite inequality holds, let us consider the sequence
\begin{equation}
\label{BN}
\Phi_{n}(x)=\chi(|x|)\mu_{j}^{-\frac{1}{2}}U_{\frac{1}{n},0}(x),
\end{equation}
where the cut-off function $\chi$ verifies:

\beq
\label{co}
 \chi(|x|)\in  \cC_{0}^{\infty}([0,\infty)),\qquad \chi(|x|)=1\ \mbox{ if } |x|\in \left[0,\frac{1}{2}\right],\qquad  \chi(|x|)=0\ \mbox{ if }|x|\geq 1.
 \eeq

Then well-known computations (see \cite{BN}) give that
\begin{equation}\label{e-2.4}
  \int_{\R^4}|\nabla \Phi_n|^{2}dx=\mu_{j}^{-1}\int_{\R^4}|\nabla U_{\frac{1}{n},0} |^{2}dx+O\left(\frac{1}{n}\right), \ \ \int_{\R^4}|\Phi_n|^{4}dx=\mu_{j}^{-2}\int_{\R^4}|U_{\frac{1}{n},0}|^{4}dx+O\left(\frac{1}{n}\right)
\end{equation}
and
\begin{equation}\label{e-2.5}
 \int_{\R^4}|\Phi_{n}|^{2}dx=O\left(\frac{1}{n}\right).
\end{equation}
%
On the other hand,  by H\"{o}lder inequality  we have for every $\rho>0$,
\begin{align*}
  \int_{\R^4}V_{j}(x)|\Phi_n|^{2}dx&=\int_{B_{\rho}(0)}V_{j}(x)|\Phi_n|^{2}dx+\int_{\R^4\setminus B_{\rho}(0)}V_{j}(x)|\Phi_n|^{2}dx \\
  &\leq \|\Phi_n\|_{L^{4}}^{2}\Big( \int_{B_{\rho}(0)}|V_{j}(x)|^{2}dx  \Big)^{\frac{1}{2}}+\|V_{j}\|_{L^2}\Big( \int_{\R^{4}\setminus B_{\rho}(0)}|\Phi_n(x)|^{4}dx  \Big)^{\frac{1}{2}}.
\end{align*}
Recalling that
\begin{equation*}
  \lim_{n\to\infty}\int_{\R^{ 4} \setminus B_{\rho}(0)}|\Phi_n|^{4}dx=0, \ \ \ \  \sup_{n\in \N}\|\Phi_n\|_{L^4}<+\infty
\end{equation*}
and
\begin{equation*}
\lim_{\rho \to 0}\int_{B_{\rho}(0)}|V_{j}(x)|^{2}dx=0,
\end{equation*}
then
$$
 \lim_{n\to\infty}\int_{\R^4}V_{j}(x)|\Phi_n|^{2}dx=0.
$$
Moreover, by \eqref{e-2.5} we have
\begin{equation}\label{e-2.7}
 \lim_{n\to\infty}\int_{\R^4}(V_{j}(x)+\lambda )|\Phi_n|^{2}dx=0.
\end{equation}
For $t_n>0$ defined by
\begin{equation*}
t_{n}^{2}=\frac{\int_{\R^4}|\nabla \Phi_n|^{2}dx+\int_{\R^4}(V_{j}(x)+\lambda )|\Phi_{n}|^2dx}{\int_{\R^4}\mu_{j}|\Phi_{n}|^{4}dx},
\end{equation*}
it turns out  $t_{n}\Phi_{n}\in \mathcal{M}_{j}$. Furthermore, in view of \eqref{e-2.4}-\eqref{e-2.7}, we deduce that   $t_n \to 1$ as $n\to \infty$. By using \eqref{e-2.4}-\eqref{e-2.7} again, we get
\begin{equation}\label{e-2.8}
  m_{j}\leq J_{j}(t_n \Phi_n)=\frac{t_{n}^{2}}{4}\Big( \int_{\R^4}|\nabla \Phi_{n}|^2dx+\int_{\R^4}(V_{j}(x)+\lambda )|\Phi_n|^{2}dx     \Big)
 = \frac{1}{4\mu_{j}}\mathcal{S}^2+o_{n}(1).
\end{equation}
Let $n\to \infty$ in \eqref{e-2.8}, then
$$m_{j}\leq \frac{1}{4\mu_{j}}\mathcal{S}^2=m_{j}^{\infty}.$$
Thus, we complete the whole proof.
\end{proof}

\begin{cor}
\label{nontrivial}
Let $(u,v)\in\cN$ be a critical point of $I$ constrained on $\cN$.
Assume that  $I(u,v)<\min\left\{\frac{1}{4\mu_{1}}\mathcal{S}^2,\frac{1}{4\mu_{2}}\mathcal{S}^2\right\}$, then $u\not\equiv0$ and $v\not\equiv 0$.
\end{cor}

Corollary \ref{nontrivial} follows directly from Lemma \ref{lm2.1}, taking into account that if  $(u,0)\in\cN$, for example,   then $u\in \cM_1$ and $I(u,0)=J_1(u)>\frac{1}{4\mu_{1}}\mathcal{S}^2$.

\begin{lem}\label{lm2.2}
  Suppose that $(u,v)\in H_{\lambda}\setminus \{(0,0)\}$, for $\lambda\ge0$, and let $\big( \tau_{(u,v)}u, \tau_{(u,v)}v \big)$, $\big(t_{(u,v)}u, t_{(u,v)} v \big)$ be the projections of $(u,v)$ on $\mathcal{N}_{\infty}$ and $\mathcal{N}$ respectively. Then
  \begin{equation*}
    \tau_{(u,v)}\leq t_{(u,v)}.
  \end{equation*}
\end{lem}
\begin{proof}
Direct calculations show that
\begin{align*}
  \tau_{(u,v)}^{2}&=\frac{\int_{\R^4}|\nabla u|^{2}dx+\int_{\R^4}|\nabla v|^{2}dx}{\int_{\R^4}(\mu_1(u^+)^4+2\beta (u^+)^{2}(v^+)^2+\mu_{2}(v^+)^4)dx} \\
  &\leq \frac{\int_{\R^4}|\nabla u|^{2}dx+\int_{\R^4}|\nabla v|^{2}dx+\int_{\R^4}(V_{1}(x)+\lambda)u^2dx+\int_{\R^4}(V_2(x)+\lambda )v^2dx}{\int_{\R^4}(\mu_1(u^+)^4+2\beta (u^+)^{2}(v^+)^2+\mu_{2}(v^+)^4)dx} \\
  & = t_{(u,v)}^{2},
\end{align*}
which implies that  $\tau_{(u,v)}\leq t_{(u,v)}$.
\end{proof}

Throughout  this paper, we always require $\beta>\max\{\mu_1,\mu_2\}$. We set
\beq
\label{b1}
k_{1}=\frac{\beta-\mu_2}{\beta^{2}-\mu_1\mu_2}, \ \ \ \ k_2=\frac{\beta-\mu_1}{\beta^{2}-\mu_1\mu_2},
\eeq
then $k_1,k_2>0$.
The next two lemmas are essentially proved  in \cite{LL} and here we omit them.

\begin{lem}(see Lemma 2.4 \cite{LL})
If $\beta>\max\{\mu_1,\mu_2 \}$, then
\beq
\label{b2}
c_{\infty}=\frac{1}{4}(k_1+k_2)\mathcal{S}^{2}.
\eeq
\end{lem}

Observe that if $\beta>\max\{\mu_1,\mu_2 \}$ then \eqref{b2} and \eqref{mj} give
\beq
\label{cinfty}
 c_\infty=\frac{1}{4}(k_1+k_2)\mathcal{S}^{2}<\min\left\{\frac{\mathcal{S}^{2}}{4\mu_1}, \frac{\mathcal{S}^{2}}{4\mu_2}\right \}=\min\{m_{1}^{\infty}, m_{2}^{\infty} \}.
 \eeq

\begin{lem}(see Lemma 2.5 \cite{LL})\label{lm2.4}
If $\beta>\max\{\mu_1,\mu_2\}$, then any nontrivial solution $(u,v)\in H_{0}$ of system \eqref{S2} must be of the form
\begin{equation*}
  (u,v)=(\sqrt{k_1}U_{\delta,y}, \sqrt{k_2}U_{\delta,y})
\end{equation*}
for some $\delta>0$ and $y\in \R^4$. Moreover, each nontrivial classical positive solution $(u,v)$ of system \eqref{S2} is a ground state solution.
\end{lem}

Next, we are verifying that under the hypotheses $(A_1)$-$(A_2)$  the infimum $c$ cannot be attained, then system \eqref{S-4} does not have any ground state solutions.
As a consequence, we need to investigate the existence of non-negative  solutions of  system \eqref{S-4} at higher energy levels.

\begin{lem}\label{lm2.6}
Let $\lambda\ge 0$, $\beta>\max\{\mu_1,\mu_2 \}$.  Suppose that $V_{1}(x), V_{2}(x)$ satisfy assumptions $(A_1)$-$(A_2)$,
 then
$c_{\infty}=c$
and $c$ cannot be attained.
\end{lem}

\begin{proof}
We define  $\tilde{\Phi}_{n}(x)=\chi(|x|)U_{\frac{1}{n},0}$, where cut-off function $\chi$ is defined in  \eqref{co}.
Then following by the arguments in the proof of Lemma \ref{lm2.1}, we get
\begin{equation*}
\lim_{n\to \infty} \Big( \int_{\R^4}(V_{1}(x)+\lambda )|\tilde{\Phi}_{n}(x)|^2dx+\int_{\R^4}(V_{2}(x)+\lambda )|\tilde{\Phi}_{n}(x)|^2dx \Big)=0.
\end{equation*}
For $t_n>0$ defined by
\begin{equation*}
t_{n}^{2}=\frac{(k_1+k_2)\int_{\R^4}|\nabla\tilde{\Phi}_n|^{2}dx+k_{1}\int_{\R^4}(V_{1}(x)+\lambda )|\tilde{\Phi}_n|^{2}dx+k_{2}\int_{\R^4}(V_{2}(x)+\lambda )|\tilde{\Phi}_n|^{2}dx}
{(k_1+k_2)\int_{\R^4}|\tilde{\Phi}_n|^{4}dx},
\end{equation*}
since
$$
k_{1}^{2}\mu_1+2\beta k_1k_2 + k_{2}^{2}\mu_2=k_1+k_2,
$$
we assert that $(t_n\sqrt{k_1} \tilde{\Phi}_n,  t_n\sqrt{k_2} \tilde{\Phi}_n )\in \mathcal{N}$
 and $t_n^{2}=1+o_{n}(1)$.
 Moreover,  by \eqref{IN},
$$
 c\leq I(t_n\sqrt{k_1} \tilde{\Phi}_n,  t_n\sqrt{k_2} \tilde{\Phi}_n)
 =\frac{t_{n}^{4}}{4}(k_1+k_2)\int_{\R^4}|\tilde{\Phi}_n|^{4}dx
 =\frac{1}{4}(k_1+k_2)\mathcal{S}^{2}+o_{n}(1).
$$

 Thus, $c\leq \frac{1}{4}(k_1+k_2)\mathcal{S}^{2}=c_{\infty}$.

For any $(u,v)\in \mathcal{N}$, let $\tau_{(u,v)}>0$ be such that $(\tau_{(u,v)}u, \tau_{(u,v)} v)\in \mathcal{N}_{\infty}$. Since $\lambda > 0$ and $V_j(x)\geq 0$ for $j=1,2$,    then by Lemma \ref{lm2.2} we get $\tau_{(u,v)}\leq 1$.
Moreover
\begin{align*}
  c_{\infty}&\leq I_{\infty}(\tau_{(u,v)}u, \tau_{(u,v)}v)=\frac{1}{4}\tau_{(u,v)}^{2}\left(\int_{\R^4}|\nabla u|^{2}dx+ \int_{\R^4}|\nabla v|^{2}dx\right) \\
  &\leq \frac{1}{4}\left(\int_{\R^4}|\nabla u|^{2}dx+ \int_{\R^4}|\nabla v|^{2}dx\right)+ \frac{1}{4}\int_{\R^4}(V_1(x)+\lambda )u^{2}dx+\frac{1}{4}\int_{\R^4}(V_2(x)+\lambda )v^{2}dx \\
  & = I(u,v).
\end{align*}
Therefore $c_{\infty}\leq c$, and thus we conclude that $c=c_{\infty}$.

To end the proof, we  show that $c$  cannot be attained. We argue by contradiction and assume that $c$ is attained by $(u_0,v_0)\in \mathcal{N}$.  It follows by Lemma \ref{lm2.2} that there exists $\tau_0\in (0,1]$ such that $(\tau_0u_0,\tau_0v_0)\in \mathcal{N}_{\infty}$.
Thus,   we have
\begin{align*}
  c_{\infty}&\leq I_{\infty}(\tau_0u_0,\tau_0v_0)\\
   &=\frac{\tau_{0}^{2}}{4}\left(\int_{\R^4}|\nabla u_0|^{2}dx+ \int_{\R^4}|\nabla v_0|^{2}dx\right)\\
   &\leq \frac{1}{4}\left(\int_{\R^4}|\nabla u_0|^{2}dx+ \int_{\R^4}|\nabla v_0|^{2}dx\right)+\frac{1}{4}\int_{\R^4}(V_{1}(x)+\lambda) |u_0|^{2}dx+\frac{1}{4}\int_{\R^4}(V_{2}(x)+\lambda)|v_0|^{2}dx\\
   &=c=c_\infty,
\end{align*}
which implies that
\begin{equation}\label{e-2.9}
\tau_0=1 \ \ \text{and} \ \ \int_{\R^4}(V_{1}(x)+\lambda)|u_0|^{2}dx+\int_{\R^4}(V_{2}(x)+\lambda)|v_0|^{2}dx=0.
\end{equation}
Moreover, $c_\infty$ is attained by $(u_0,v_0)\in\cN_\infty$ and
by \eqref{cinfty} and   Lemma \ref{lm2.4}
we   deduce  $u_0 >0$ and $v_0> 0$.
Then
\begin{equation*}
  \int_{\R^4}(V_1(x)+\lambda)|u_0|^{2}dx+\int_{\R^4}(V_2(x)+\lambda)|v_0|^2dx >0,
\end{equation*}
which is in contradiction with \eqref{e-2.9}.
Thus,
$c$ is not attained.
\end{proof}

We remark that in the first part of Theorem \ref{Th1.2}, the hypothesis
  $(A_1)$ is necessary, due to the non-existence result below.

\begin{lem}\label{lm2.5}
Let $\lambda>0$ and suppose that $(u,v)$ is  a solution of
\begin{equation}\label{S3}
  \begin{cases}
    -\Delta u+ \lambda u=\mu_1u^{3}+\beta uv^{2}, \ \  x\in \R^4, \\
    -\Delta v+ \lambda v=\mu_2v^{3}+\beta vu^{2}, \ \  x\in \R^4.
  \end{cases}
\end{equation}
Then $(u,v)=(0,0)$.
\end{lem}
\begin{proof}
We will prove the lemma by a Pohozaev identity. More precisely, our proof follows from a classical strategy of testing the equation against $x\cdot \nabla u$, which is made rigorous by multiplying by cut-off functions, see  \cite[Appendix B]{WM}.

We assume that  $(u,v)$ is a solution of \eqref{S3}.
Let $\tilde{\phi}(x)\in \cC_{0}^{1}(\R^4)$ be a cut-off function such that $\tilde{\phi}(x)=1$ on $B_{1}(0)$.
Multiplying the first equation by $u_{\rho}=\tilde{\phi}(\rho x)(x\cdot \nabla u)$, the second equation by $v_{\rho}=\tilde{\phi}(\rho x)(x\cdot\nabla v)$, summing up and integrating by parts,  we have
   \begin{equation}\label{e-2.10}
   \begin{split}
&     \int_{\R^4}\nabla u \cdot \nabla u_{\rho}\, dx+ \lambda \int_{\R^4}uu_{\rho}\,dx
+ \int_{\R^4}\nabla v \cdot \nabla v_{\rho}\, dx+\lambda \int_{\R^4}vv_{\rho}\, dx\\
&=\mu_1\int_{\R^4}u^{3}u_{\rho}\,dx
  + \mu_2\int_{\R^4}v^{3}v_{\rho}\, dx+\beta\int_{\R^4} \tilde{\phi}(\rho x) \, x \cdot\D\left( \frac{u^2v^2}{2}\right)  \, dx.
 \end{split}
   \end{equation}
Direct calculations show that
\begin{align*}
  \int_{\R^4}\nabla u \cdot \nabla u_{\rho}\, dx&=\int_{\R^4}\tilde{\phi}(\rho x)\left(  |\nabla u|^{2}+ x\cdot \nabla\left(\frac{|\nabla u|^{2}}{2}\right)\right)dx+\int_{\R^4}\big(\nabla u\cdot\nabla \tilde{\phi}(\rho x)\big)  (\rho x\cdot \nabla u)\, dx  \\
  &=-\int_{\R^4}\Big( 2\tilde{\phi}(\rho x)+ \rho x \cdot \nabla \tilde{\phi} (\rho x) \Big)\frac{|\nabla u|^{2}}{2}\, dx +\int_{\R^4}\big(\nabla u\cdot\nabla \tilde{\phi}(\rho x)\big)  (\rho x\cdot \nabla u)\, dx.
\end{align*}
Using the Lebesgue's dominated convergence theorem, we have
\begin{equation}\label{add-1}
  \lim_{\rho \to 0} \int_{\R^4}\nabla u \cdot \nabla u_{\rho} dx=-\int_{\R^4}|\nabla u|^{2}.
\end{equation}
Next,
\begin{align*}
  \lambda\int_{\R^4}uu_{\rho}dx&= \lambda \int_{\R^4}u(x)\tilde{\phi}(\rho x)\,\big(x\cdot \nabla u\big)\, dx\\
  &= \lambda\int_{\R^4}\tilde{\phi}(\rho x)\,\left(x\cdot \nabla \left(\frac{|u|^{2}}{2}\right)\right)\,dx  \\
  &=  -\lambda  \int_{\R^4}\big(4 \tilde{\phi}(\rho x)+\rho x \cdot \nabla \tilde{\phi}(\rho x)\big)\frac{|u(x)|^{2}}{2}dx.
\end{align*}
By Lebesgue's dominated convergence theorem again, it holds that
\begin{equation}\label{e-2.12}
\lim_{\rho\to0} \lambda \int_{\R^4}u u_{\rho}dx=-2 \lambda \int_{\R^4}|u|^{2}dx.
\end{equation}
Similarly,  we can also prove
\begin{equation}\label{e-2.13bis}
 \lim_{\rho \to 0} \beta \int_{\R^4}  \tilde{\phi}(\rho x) \, x \cdot\D\left( \frac{u^2v^2}{2}\right)  \, dx
=-2\beta \int_{\R^4} u^2v^2  dx
\end{equation}
and
\begin{equation}
\lim_{\rho \to 0}\mu_{1}\int_{\R^4}u^{3}u_{\rho}dx = -\mu_{1}\int_{\R^4}|u|^{4}dx.
\end{equation}
Repeat the arguments as above, we can also prove that
\begin{equation}
 \lim_{\rho \to 0} \int_{\R^4}\nabla v \cdot \nabla v_{\rho} dx=-\int_{\R^4}|\nabla v|^{2},\quad
\end{equation}
and
\begin{equation}\label{e-2.13}
\lim_{\rho\to0}  \lambda \int_{\R^4}v v_{\rho}dx=-2 \lambda \int_{\R^4}|v|^{2}dx, \qquad \lim_{\rho \to 0}\mu_{2}\int_{\R^4}v^{3}v_{\rho}dx = -\mu_{2}\int_{\R^4}|v|^{4}dx.
\end{equation}

 Thus, by \eqref{add-1} -- \eqref{e-2.13}, we  get
 \begin{equation}\label{e-2.17}
  \int_{\R^4}(|\nabla u|^{2}+|\nabla v|^{2})dx+ 2\lambda\int_{\R^4}(|u|^{2}+ |v|^2)dx=\mu_1\int_{\R^4}|u|^{4}dx+\mu_2\int_{\R^4}|v|^{4}dx+2\beta\int_{\R^4} u^2v^2dx.
 \end{equation}
On the  other hand, since $(u,v)$ is a pair solution of  system \eqref{S3}, then
\begin{equation}\label{e-2.18}
   \int_{\R^4}(|\nabla u|^{2}+|\nabla v|^{2})dx+\lambda\int_{\R^4}(|u|^{2}+ |v|^2)dx=\mu_1\int_{\R^4}|u|^{4}dx+\mu_2\int_{\R^4}|v|^{4}dx+2\beta\int_{\R^4} u^2v^2dx.
\end{equation}
Thus, it follows by \eqref{e-2.17} and \eqref{e-2.18} that
\begin{equation*}
  \lambda\int_{\R^N} (|u|^{2}+|v|^{2})dx=0.
\end{equation*}
Since  $\lambda>0$,  we have $(u,v)=(0,0)$.
\end{proof}

\vspace{0.2cm}

\section{A global compactness result}\label{s3}

\quad In the present paper, some of our arguments are closely related to the lack of compactness question. In order to overcome this difficulty, we introduce a global compactness result for the coupled system \eqref{S-4} with $\lambda> 0$. See \cite{AFM,BC,CP,CP2,GL,MW,S}
 for similar results on scalar equations and \cite{LL,PPW} on a coupled Schr\"{o}dinger system. However, we strongly point out that, in the present paper, we face different situations, some new delicate estimates concerning the strictly positive potentials and coupled terms are necessary.
The main result is the following:
\begin{thm}\label{th3.1}
Assume that $\{(u_n,v_n)\}\subset H_{\lambda}$ is a Palais-Smale sequence  for the functional $I$ at    level $d$.
Then  there exist a solution $(u^0,v^0)$ of system \eqref{S-4}, $\ell$ sequences of positive numbers $\{\sigma_{n}^{k}\}$ $(1\leq k\leq \ell)$ and $\ell$  sequences of points $\{{{y}}_{n}^{k}\}$ $(1\leq k\leq \ell)$ in $\R^4$, such that, up to  a subsequence,
\begin{equation} \label{e-3.1}
  \|(u_n,v_n)\|_{H_\lambda}^{2}=\|(u^0,v^0)\|_{H_{\lambda}}^{2}+\sum_{k=1}^{\ell}\left\|\Big((\sigma_{n}^{k})^{-1}u^k\left(\frac{\cdot-{{y}}_{n}^{k}}{\sigma_{n}^{k}}\right),\ (\sigma_{n}^{k})^{-1}v^k\left(\frac{\cdot-{{y}}_{n}^{k}}{\sigma_{n}^{k}}\right)\Big)\right\|_{H_0}^{2}+o_{n}(1)
\end{equation}
and
\begin{equation} \label{e-3.2}
I(u_n,v_n)=I(u^0,v^0)+\sum_{k=1}^{\ell}I_{\infty}(u^k,v^k)+o_{n}(1),
\end{equation}
where $\sigma_{n}^{k} \to 0$ as $n \to \infty$ and $(u^k,v^k)\neq (0,0)$ solve system \eqref{S2}.
\end{thm}

Before proving Theorem \ref{th3.1}, we  first introduce the following  Br\'{e}zis-Lieb type lemma and its variants, which are proved in  \cite[Lemma 3.4]{MW}
and in \cite[Lemma 3.2]{LL} by using some ideas introduced by Br\'{e}zis and Lieb in \cite{BL}.

\begin{lem}\label{lm3.2}
$(i)$ Suppose that $u_n\rightharpoonup u$ in $L^4(\R^4)$ and $u_n \to u$ a.e. in $\R^4$, then
\begin{align*}
  & \lim_{n\to \infty} \Big(\int_{\R^4} (u_n^+)^{4}dx - \int_{\R^4}((u_n-u)^{+})^{4}dx \Big)= \int_{\R^4}(u^+)^{4} dx; \\
  &\lim_{n \to \infty }\int_{\R^4}\Big| (u_n^{+})^{3}-((u_n-u)^{+})^{3}-(u^+)^{3} \Big|^{\frac{4}{3}}dx=0.
\end{align*}
$(ii)$ Suppose that $(u_n,v_n)\rightharpoonup (u,v)$ in $L^{4}(\R^4)\times L^{4}(\R^4)$ and  $(u_n,v_n) \to (u,v)$ a.e. in $\R^4$, then
\begin{align*}
  & \lim_{n\to \infty} \Big(\int_{\R^4}(u_n^{+})^2(v_n^+)^{2}dx-\int_{\R^4}((u_n-u)^+)^{2}((v_n-v)^+)^2dx \Big)  = \int_{\R^4}(u^+)^2(v^+)^2dx; \\
  & \lim_{n\to \infty} \int_{\R^4}\Big((u_n)^+(v_n^{+})^{2}-(u_n-u)^{+}((v_n-v)^{+})^{2}-(u^+)(v^+)^{2}   \Big|^{\frac{4}{3}}dx=0.
\end{align*}
\end{lem}



\textbf{Proof of Theorem 3.1} \,\,
In this proof, we argue up to suitable subsequences.
Since $\{(u_n,v_n)\}\subset H_{\lambda}$ is a Palais-Smale sequence  for the functional $I$ at level $d$, then it is easy to prove that $\{(u_n,v_n)\}$ is bounded in $H_{\lambda}$.
We may suppose that $(u_n,v_n) \rightharpoonup (u^0,v^0)$ in $H_{\lambda}$, then $(u_n,v_n)\to (u^0,v^0)$ a.e. in $\R^4$ and $(u_n,v_n)\to (u^0,v^0)$ in $L_{\lloc}^{2}(\R^4) \times L_{\lloc}^{2}(\R^4) $.
So we can infer that $(u^0,v^0)$ is a solution of system \eqref{S-4}.
Let us define $(u_n^{1},v_{n}^{1})=(u_n,v_n)-(u^0,v^0)$, then
\beq
\label{h3.3}
(u_n^{1},v_{n}^{1})\rightharpoonup (0,0)\ \mbox{ in }H_{\lambda},\quad  (u_{n}^{1},v_{n}^{1})\to (0,0) \  \mbox{ a.e.  in }   \R^4 \ \text{and in} \ L_{loc}^{2}(\R^4) \times L_{loc}^{2}(\R^4).
\eeq
From Lemma \ref{lm3.2} and  \cite[Lemma 2.13]{WM}, we deduce that
\begin{equation*}
I_{\lambda,\infty}(u_{n}^{1},v_{n}^{1})=I(u_n,v_n)-I(u^0,v^0)+o_{n}(1)
\end{equation*}
and
\begin{equation}
 I'_{\lambda,\infty}(u_{n}^{1},v_{n}^{1})=I'(u_n,v_n)-I'(u^0,v^0)+o_{n}(1)=o_{n}(1),
\end{equation}
where
\begin{equation}
  I_{\lambda,\infty}(u,v):=I_{\infty}(u,v)+\frac{\lambda}{2}\int_{\R^4}u^{2}dx+\frac{\lambda}{2}\int_{\R^4}v^2dx.
\end{equation}
Therefore, $\{(u_n^{1},v_{n}^{1})\}$ is a Palais-Smale sequence for $I_{\lambda,\infty}$ at level $d-I(u^0,v^0)$.

If $(u_n^{1},v_{n}^{1})\to (0,0)$ in $H_{\lambda}$, then we are done.
If $(u_n^{1},v_{n}^{1})\rightharpoonup (0,0)$, but $(u_n^{1},v_{n}^{1})\nrightarrow (0,0)$ in $H_{\lambda}$, then there exists a positive constant $\tilde{c}>0$ such that
\begin{equation*}
  \|(u_{n}^{1},v_{n}^{1})\|_{H_{\lambda}}\geq \tilde{c}>0,
\end{equation*}
that is
\begin{equation*}
  \int_{\R^4}\big(|\nabla u_n^1|^{2}+\lambda  |u_n^1|^2\big)dx+\int_{\R^4}\big(|\nabla v_n^1|^{2}+\lambda  |v_{n}^{1}|^{2}\big)dx\geq \tilde{c}.
\end{equation*}
Since $\{(u_n^{1},v_{n}^{1})\}$ is a Palais-Smale sequence for $I_{\lambda,\infty}$, then by H\"{o}lder inequality and Young inequality, we have
\begin{align*}
\tilde{c}&\leq \int_{\R^4}\big(|\nabla u_n^1|^{2}+\lambda  |u_n^1|^2\big)dx+\int_{\R^4}\big(|\nabla v_n^1|^{2}+\lambda  |v_{n}^{1}|^{2}\big)dx \\
&=\int_{\R^4}\mu_1 ((u_n^1)^+)^{4}dx+2\beta\int_{\R^4} ((u_n^1)^+)^{2}((v_n^1)^+)^2dx+\int_{\R^4}\mu_2((v_{n}^{1})^+)^{4}dx+o_{n}(1) \\
& \leq C \Big(\int_{\R^4}((u_{n}^{1})^+)^{4}dx+\int_{\R^4}((v_n^1)^+)^4dx\Big)+o_{n}(1).
\end{align*}
Thus, there exists  a positive constant $\hat{c}>0$ such that
\begin{equation}
  \int_{\R^4}((u_{n}^{1})^+)^{4}dx+ \int_{\R^4}((v_n^1)^+)^4 dx\geq \hat{c}.
\end{equation}

Let $\|(v_{n}^{1})^+\|_{L^{4}}\neq 0$, $\forall n\in\N$ and set
\begin{equation*}
  t_{n}=\frac{\|(u_n^1)^{+}\|_{L^{4}}^{2}}{\|(v_{n}^{1})^+\|_{L^{4}}^2}.
\end{equation*}
Then by calculation we have
\begin{align*}
  \mathcal{S}(t_n+1)\|(v_n^1)^{+}\|_{L^{4}}^{2}&= \mathcal{S}(\|(u_n^1)^{+}\|_{L^{4}}^{2}+\|(v_n^1)^{+}\|_{L^{4}}^{2}) \\
  &\leq  \|(u_n^1, v_{n}^{1})\|_{H_{\lambda}}^{2}  \\
  &= \mu_1\int_{\R^4}|(u_n^1)^+|^{4}dx+\mu_2\int_{\R^4}|(v_n^1)^+|^{4}dx+2\beta\int_{\R^4} |(u_{n}^{1})^{+}|^2|(v_{n}^{1})^+|^2dx+o_{n}(1) \\
  &\leq \Big( \mu_1 \|(u_n^1)^+\|_{L^4}^{4}+ \mu_2 \|(v_n^1)^+\|_{L^4}^{4}+2\beta \|(u_n^1)^+\|_{L^4}^{2}\|(v_n^1)^+\|_{L^4}^{2} \Big)+o_{n}(1) \\
  &= \Big( \mu_1 t_{n}^{2}+ \mu_2 +2\beta t_n \Big) \|(v_n^1)^+\|_{L^4}^{4}+o_{n}(1),
\end{align*}
which leads to
\begin{equation*}
\|(v_n^1)^{+}\|_{L^{4}}^{2}\geq \frac{\mathcal{S}(t_n+1)-o_{n}(1)}{\mu_1 t_{n}^{2}+ \mu_2 +2\beta t_n }.
\end{equation*}
A direct computation shows that
$$
\inf_{t\ge 0}\frac{(1+t)^2}{\mu_1 t^2+2\beta t+\mu_2}=k_1+k_2
$$
(see  \cite[Lemma 2.3]{LL}), then
\begin{align}\label{guo-3.8}
  \int_{\R^4}(|\nabla u_{n}^{1}|^{2}+|\nabla v_{n}^{1}|^{2})dx & \geq \mathcal{S}(\|(u_n^1)^{+} \|_{L^4}^{2}+\|(v_n^1)^{+} \|_{L^4}^{2}) \nonumber \\
  &=  \mathcal{S}(t_n+1)\|(v_{n}^{1})^{+}\|_{L^4}^{2} \nonumber \\
  & \geq \frac{\mathcal{S}^2(t_n+1)^2}{\mu_1 t_{n}^{2}+ \mu_2 +2\beta t_n }-o_{n}(1) \\
  & \geq (k_1+k_2)\mathcal{S}^{2}-o_{n}(1) \nonumber  \\
  & = 4c_{\infty}-o_{n}(1). \nonumber
\end{align}

Next, we denote by $P_i$ the hypercubes with disjoint interior and unitary sides such that
$\mathbb{R}^4=\cup_{i\in\mathbb{N}}P_i$,  and then we also define
$$
d_n^1:=\max_{i\in\mathbb{N}}\Big(\int_{P_i}\big[((u_n^1)^+)^4+((v_n^1)^+)^4\big]dx\Big)^{\frac{1}{4}}.
$$
By calculation, we have
\begin{align*}
\hat{c}&\leq (\|(u_n^1)^+\|_{L^4}^4+\|(v_n^1)^+\|_{L^4}^4)=\sum_{i=1}^\infty (\|(u_n^1)^+\|_{L^4(P_i)}^4+\|(v_n^1)^+\|_{L^4(P_i)}^4)\\
&\leq (d_n^1)^2\sum_{i=1}^\infty (\|(u_n^1)^+\|_{L^4  (P_i)}^2+\|(v_n^1)^+\|_{L^4(P_i)}^2) \\
&\leq  (d_n^1)^2\cdot \widehat S\sum_{i=1}^\infty(\| u_n^1\|_{H^1 (P_i)}^2+\|v_n^1\|_{H^1 (P_i)}^2)\\
&=(d_n^1)^2\cdot  \widehat S (\|u_n^1\|_{H^1}^2+\|v_n^1\|_{H^1}^2) \\
&\leq (d_n^1)^2\cdot \widehat S C,
\end{align*}
where $\widehat S$ is a constant independent of $i$ and the last inequality is due to the boundedness of $(u_n^1, v_n^1)$ in $H_{\lambda}$.
Thus, there exists a positive constant $a$ such that
\begin{equation}
\label{h3.7}
d_n^1\geq a>0\qquad\forall n\in\N.
\end{equation}
We observe also that
\beq
\label{ciompa}
\lim_{n\to\infty}\max_{i\in\mathbb{N}}\int_{P_i}\big[(u_n^1)^2+(v_n^1)^2\big]dx=0.
\eeq
Indeed, if \eqref{ciompa} is not true, and   $\max_{i\in\mathbb{N}}\int_{P_i}\big[(u_n^1)^2+(v_n^1)^2\big]dx$ is achieved in a cube $P_{i_n}$ centered in $z_n$, then the sequence $\{(u_n^1(\cdot -z_n),v_n^1(\cdot -z_n))\}$ converges to a nonzero solution of  \eqref{S3}, contrary to Lemma \ref{lm2.5}.

Since there exists a constant $\widetilde C>0$, independent of $i$, such that
\begin{align*}
\int_{P_i}(|\nabla u_n^1|^2+|\nabla v_n^1|^2)dx+\lambda \int_{P_i}((u_n^1)^2+( v_n^1)^2)dx&\geq \widetilde C (\|(u_n^1)^+\|_{L^4(P_i)}^2+\|(v_n^1)^+\|_{L^4(P_i)}^2)\\
&\geq \widetilde C (\|(u_n^1)^+\|_{L^4(P_i)}^4+\|(v_n^1)^+\|_{L^4(P_i)}^4)^{\frac{1}{2}},
\end{align*}
then, by  \eqref{h3.7} and \eqref{ciompa}, for large $n$ we get
$$
\max_{i\in\mathbb{N}}\int_{P_i}(|\nabla u_n^1|^2+|\nabla v_n^1|^2)dx
  \geq\frac{ \widetilde C}{2}(d_n^1)^2
\geq \frac{1}{2} \widetilde C a^2.
$$

Define the Levy concentration function
\begin{equation}\label{n1}
Q_{n}(r)=\sup_{y\in\mathbb{R}^4}\int_{B_{r}(y)}(|\nabla u_n^1|^2+|\nabla v_n^1|^2)\, dx.
\end{equation}
 Taking into  account $Q_n(0)=0$, $Q_n(\infty)\geq   \frac{1}{2} \widetilde C a^2  $ (also $Q_n(\infty)\geq 4c_{\infty}$), and $Q_n(r)$ continuous in $r>0$, we can find  $y_n\in\mathbb{R}^4$ and $\sigma_n>0$ such that
\begin{equation}\label{n2}
Q_n(\sigma_n)=\int_{B_{\sigma_n}(y_n)}(|\nabla u_n^1|^2+|\nabla v_n^1|^2)dx=\hat{\delta}<\,\min\Big\{\frac{2c_\infty}{L},     \frac{\widetilde C a^2}{4}\Big\},
\end{equation}
where $L$ is the least number of balls with radius 1 covering a ball of radius 2 and $\hat{\delta}>0$ is independent of $n$.
The radii $\{\sigma_n\}$ are bounded, otherwise, for large $n$,

$$
Q_n(\sigma_n)\geq \max_{i\in\mathbb{N}}\int_{P_i}(|\nabla u_n^1|^2+|\nabla v_n^1|^2)dx\geq  \frac{ \widetilde C a^2}{2},
$$
which yields to a contradiction.

Let
\begin{equation}
\label{mam}
  (\hat{u}_{n}^{1}, \hat{v}_{n}^{1}):=( \sigma_{n}u_{n}^{1}(\sigma_n x +y_n),  \sigma_{n}v_{n}^{1}(\sigma_n x+y_n)),
\end{equation}
then
\begin{equation*}
  \int_{\R^4}(|\nabla\hat{u}_{n}^{1}|^2+|\nabla \hat{v}_{n}^{1}|^2)\,dx=\int_{\R^4} (|\nabla u_{n}^{1}|^2+|\nabla v_{n}^{1}|^2)\,dx<\infty
\end{equation*}
and
\begin{equation}\label{guo-add-new}
  \int_{B_{1}(0)}(|\nabla \hat{u}_{n}^{1}|^2+|\nabla \hat{v}_{n}^{1}|^2)dx=\hat{\delta}.
\end{equation}
Hence, there exists $(u^1,v^1)\in H_0$ such that $(\hat{u}_{n}^{1}, \hat{v}_{n}^{1})\rightharpoonup (u^{1},v^{1})$ in $H_0$ and a.e. on $\R^4 \times \R^4 $.

Next, we will show that $(u^{1},v^{1}) \neq (0,0)$ and $(u^{1},v^{1})$ is a nontrivial solution of the limit problem \eqref{S2}.
In fact, arguing as in \cite{S} (see also  \cite[Lemma 3.6]{PPW}), we can find $\rho\in [1,2]$ such that
$\hat{u}_{n}^{1}-u^1\to 0$ and $\hat{v}_{n}^{1}-v^1\to 0$ in $H^{1/2,2}(\partial B_\rho(0))$.
Then, the solutions $\phi_{1,n}$, $\phi_{2,n}$ of
\begin{equation*}
  \begin{cases}
    -\Delta \phi=0  \qquad  x\in B_{3}(0) \setminus B_{\rho}(0), \\
    \phi|_{\partial B_{\rho}(0)}= \hat{u}_{n}^{1}-u^1, \ \ \phi|_{\partial B_{3}(0)}=0,
  \end{cases}
\end{equation*}
and
\begin{equation*}
  \begin{cases}
    -\Delta \phi=0  \qquad  x\in B_{3}(0) \setminus B_{\rho}(0), \\
    \phi|_{\partial B_{\rho}(0)}= \hat{v}_{n}^{1}-v^1, \ \ \phi|_{\partial B_{3}(0)}=0,
  \end{cases}
\end{equation*}
respectively, satisfy
\begin{equation}\label{guo-3.9}
  \phi_{1,n}\to 0, \ \ \ \ \phi_{2,n}\to 0 \ \ \ \text{in} \ \ H^{1}(B_{3}(0)\setminus B_{\rho}(0)).
\end{equation}
Let
\begin{equation}\label{guo-3.10}
 \varphi_{1,n} =
 \left\{
  \begin{array}{ll}
    \hat{u}_{n}^{1}-u^1 & \quad x\in B_{\rho}(0), \\
    \phi_{1,n} & \quad x\in B_{3}(0) \setminus B_{\rho}(0),\\
    0 & \quad x\in \R^4 \setminus B_{3}(0)
  \end{array}
\right.
\end{equation}
and
\begin{equation}\label{guo-3.11}
 \varphi_{2,n}(x)=
 \left\{
  \begin{array}{ll}
    \hat{v}_{n}^{1}-v^1 & \quad  x\in B_{\rho}(0), \\
    \phi_{2,n} & \quad  x\in B_{3}(0) \setminus B_{\rho}(0), \\
    0 & \quad  x\in \R^4 \setminus B_{3}(0),
 \end{array}
\right.
\end{equation}
then we have
\begin{equation}
\label{h3.14}
  \|  \varphi_{j,n} \|_{L^{2}(\R^4)} \to 0, \ \ \ \ \text{as} \ \ n\to\infty, \ \ j=1,2,
\end{equation}
 because $\|\varphi_{j,n}\|_{H^1(\R^4)}=\|\varphi_{j,n}\|_{H^1_0(B_3(0)}\le C$ and $\varphi_{j,n}\to 0$ a.e. in $\R^4$.
Set
\begin{equation*}
  \hat{\varphi}_{j,n}=\sigma_{n}^{-1} \varphi_{j,n} (\frac{x}{\sigma_n}), \ \ \ \ j=1,2.
\end{equation*}
By \eqref{guo-3.9}-\eqref{h3.14}, we get
\begin{equation}\label{guo-3.12}
  \|\hat{\varphi}_{1,n} \|_{H^1}^2=\| \varphi_{1,n}\|_{D^{1,2}}^2+\lambda\sigma_{n}^{2}\| \varphi_{1,n}\|_{L^{2}}^2+o_{n}(1)=\|\hat{u}_{n}^{1}-u^1\|_{D^{1,2}(B_{\rho}(0))}^2+o_{n}(1).
\end{equation}
Similarly,
\begin{equation}\label{guo-3.13}
   \|\hat{\varphi}_{2,n} \|_{H^1}^2=\|\hat{v}_{n}^{1}-v^1\|_{D^{1,2}(B_{\rho}(0))}^2+o_{n}(1).
\end{equation}
Since $\{({u}_{n}^{1}, {v}_{n}^{1}  )\}$ is a Palais-Smale sequence of $I_{\lambda,\infty}$, then
\begin{equation}\label{guo-3.14}
 \langle I'_{\infty}(\hat{u}_{n}^{1},\hat{v}_{n}^{1} ), (\varphi_{1,n}, \varphi_{2,n})\rangle
 =  \langle I'_{\lambda,\infty}({u}_{n}^{1},{v}_{n}^{1} ), (\hat{\varphi}_{1,n}, \hat{\varphi}_{2,n})\rangle+o_{n}(1)=o_{n}(1).
\end{equation}
On the one hand, from \eqref{guo-3.9}, \eqref{guo-3.14} and Lemma \ref{lm3.2} we obtain
\begin{align*}
o_{n}(1)&= \langle I'_{\infty}(\hat{u}_{n}^{1},\hat{v}_{n}^{1} ), (\varphi_{1,n}, \varphi_{2,n})\rangle   \nonumber \\
 &= \int_{B_{\rho}(0)} \nabla  \hat{u}_{n}^{1} \nabla(\hat{u}_{n}^{1}-u^1)dx+ \int_{B_{\rho}(0)} \nabla  \hat{v}_{n}^{1} \nabla(\hat{v}_{n}^{1}-v^1)dx  \nonumber  \\
 & \ \ - \mu_{1} \int_{B_{\rho}(0)} |(\hat{u}_{n}^{1})^+|^2 (\hat{u}_{n}^{1})^{+}(\hat{u}_{n}^{1}-u^1)dx- \mu_{2} \int_{B_{\rho}(0)} |(\hat{v}_{n}^{1})^+|^2 (\hat{v}_{n}^{1})^{+}(\hat{v}_{n}^{1}-v^1)dx \\
 & \ \ - \beta \int_{B_{\rho}(0)}  (\hat{u}_{n}^{1})^{+}(\hat{u}_{n}^{1}-u^1) |(\hat{v}_{n}^{1})^+|^2 dx- \beta \int_{B_{\rho}(0)}  |(\hat{u}_{n}^{1})^+|^2 (\hat{v}_{n}^{1})^{+}(\hat{v}_{n}^{1}-v^1)  dx  +o_n(1)\nonumber  \\
 &= \int_{B_{\rho}(0)} | \nabla(\hat{u}_{n}^{1}-u^1)|^2 dx+ \int_{B_{\rho}(0)} |\nabla(\hat{v}_{n}^{1}-v^1)|^{2}dx  \nonumber  \\
 & \ \ - \mu_{1} \int_{B_{\rho}(0)} |(\hat{u}_{n}^{1} - u^1)^+|^4 dx- \mu_{2} \int_{B_{\rho}(0)} |(\hat{v}_{n}^{1} -v^1)^{+}|^4 dx  \nonumber \\
 & \ \ - 2\beta \int_{B_{\rho}(0)}  |(\hat{u}_{n}^{1}-u^1)^{+}|^{2} |(\hat{v}_{n}^{1} - v^1)^{+}|^2 dx+o_n(1). \ \nonumber
\end{align*}
Moreover, by \eqref{guo-3.9} and the scale invariance, we get
\begin{align}\label{guo-3.16}
o_{n}(1)&= \int_{\R^4} | \nabla \varphi_{1,n}|^2 dx+ \int_{\R^4} |\nabla \varphi_{2,n}|^{2}dx
 - \mu_{1} \int_{\R^4} |(\varphi_{1,n})^{+}|^4 dx- \mu_{2} \int_{\R^4} |(\varphi_{2,n})^+|^4 dx  \nonumber \\
 & \ \ - 2\beta \int_{\R^4}  |(\varphi_{1,n})^{+}|^{2} |(\varphi_{2,n})^{+}|^2 dx    \\
& = \int_{\R^4} | \nabla \hat{\varphi}_{1,n}|^2 dx+ \int_{\R^4} |\nabla \hat{\varphi}_{2,n}|^{2}dx
 - \mu_{1} \int_{\R^4} |(\hat{\varphi}_{1,n})^{+}|^4 dx- \mu_{2} \int_{\R^4} |(\hat{\varphi}_{2,n})^+|^4 dx  \nonumber  \\
 & \ \ - 2\beta \int_{\R^4}  |(\hat{\varphi}_{1,n})^{+}|^{2} |(\hat{\varphi}_{2,n})^{+}|^2 dx.  \nonumber
\end{align}
If $\big((\hat{\varphi}_{1,n})^{+}, (\hat{\varphi}_{2,n})^{+}\big)\nrightarrow (0,0)$ in $H_{\lambda}$, we define $t_n>0$  by
\begin{equation*}
  t_{n}^{2}=\frac{\|(\hat{\varphi}_{1,n},\hat{\varphi}_{2,n})\|_{H_0}^2}{ \mu_{1} \int_{\R^4} |(\hat{\varphi}_{1,n})^{+}|^4 dx+ \mu_{2} \int_{\R^4} |(\hat{\varphi}_{2,n})^+|^4 dx+ 2\beta \int_{\R^4}  |(\hat{\varphi}_{1,n})^{+}|^{2} |(\hat{\varphi}_{2,n})^{+}|^2 dx}.
\end{equation*}
Then $(t_n\hat{\varphi}_{1,n}, t_n\hat{\varphi}_{2,n} )\in \mathcal{N}_{\infty}$ and moreover,
\begin{align*}
c_{\infty}&\leq I_{\infty}(t_n\hat{\varphi}_{1,n}, t_n\hat{\varphi}_{2,n} )=\frac{1}{4}t_{n}^{2}\int_{\R^4}(|\nabla \hat{\varphi}_{1,n}|^{2}+|\nabla \hat{\varphi}_{2,n}|^{2})dx    \nonumber \\
&= \frac{1}{4}\frac{\|(\hat{\varphi}_{1,n},\hat{\varphi}_{2,n})\|_{H_0}^4}{ \mu_{1} \int_{\R^4} |(\hat{\varphi}_{1,n})^{+}|^4 dx+ \mu_{2} \int_{\R^4} |(\hat{\varphi}_{2,n})^+|^4 dx+ 2\beta \int_{\R^4}  |(\hat{\varphi}_{1,n})^{+}|^{2} |(\hat{\varphi}_{2,n})^{+}|^2 dx},
\end{align*}
which implies that
\beq
\label{en}
\begin{split}
  & \hspace{-1cm} \mu_{1} \int_{\R^4} |(\hat{\varphi}_{1,n})^{+}|^4 dx+ \mu_{2} \int_{\R^4} |(\hat{\varphi}_{2,n})^+|^4 dx+ 2\beta \int_{\R^4}  |(\hat{\varphi}_{1,n})^{+}|^{2}  |(\hat{\varphi}_{2,n})^{+}|^2 dx \\
  & \leq \frac{1}{4c_{\infty}} \left(\int_{\R^4}(|\nabla \hat{\varphi}_{1,n}|^{2}+|\nabla \hat{\varphi}_{2,n}|^{2})dx\right)^{2}.
\end{split}
\eeq
On the other hand, by calculation we have
\begin{align}\label{guo-3.15}
 & \hspace{-1cm} \int_{\R^4}(|\nabla \hat{\varphi}_{1,n}|^{2}+ |\nabla \hat{\varphi}_{2,n}|^{2})dx   \nonumber \\
 & = \int_{\R^4}(|\nabla \varphi_{1,n}|^{2}+ |\nabla \varphi_{2,n}|^{2})dx  \nonumber \\
 &= \int_{B_{\rho}(0)}|\nabla (\hat{u}_{n}^{1}-u^1)|^{2}dx+ \int_{B_{\rho}(0)}|\nabla (\hat{v}_{n}^{1}-v^1)|^{2}dx +o_{n}(1)  \nonumber \\
  &= \int_{B_{\rho}(0)}|\nabla \hat{u}_{n}^{1}|^{2}dx-\int_{B_{\rho}(0)}|\nabla u^1|^{2}dx+ \int_{B_{\rho}(0)}|\nabla \hat{v}_{n}^{1}|^{2}dx-\int_{B_{\rho}(0)}|\nabla v^1|^{2}dx+o_{n}(1)  \nonumber \\
  & \leq \int_{B_{\rho}(0)}|\nabla \hat{u}_{n}^{1}|^{2}dx+  \int_{B_{\rho}(0)}|\nabla \hat{v}_{n}^{1}|^{2}dx  +o_{n}(1).
\end{align}
Then by  \eqref{guo-3.16}-\eqref{guo-3.15}, we get
\begin{align*}
  o_{n}(1) \geq \left(1-\frac{1}{4c_{\infty}} \int_{B_{\rho}(0)}(|\nabla \hat{u}_{n}^1|^{2}+|\nabla \hat{v}_{n}^1|^{2})dx\right) \cdot \int_{\R^4}(|\nabla \hat{\varphi}_{1,n}|^{2}+|\nabla \hat{\varphi}_{2,n}|^{2})dx.
\end{align*}
Note that from \eqref{n1}, \eqref{n2} and \eqref{guo-add-new} we infer
\begin{equation*}
  \int_{B_{\rho}(0)}(|\nabla \hat{u}_{n}^1|^{2}+|\nabla \hat{v}_{n}^1|^{2})dx\leq L \int_{B_{1}(0)}(|\nabla \hat{u}_{n}^1|^{2}+|\nabla \hat{v}_{n}^1|^{2})dx<2c_{\infty}.
\end{equation*}
Thus for $n\to \infty$,
\begin{equation*}
  \int_{\R^4}(|\nabla \hat{\varphi}_{1,n}|^{2}+|\nabla \hat{\varphi}_{2,n}|^{2})dx \to 0.
\end{equation*}
Then, from \eqref{guo-3.12} and \eqref{guo-3.13} we infer
\begin{equation}
\label{mam2}
  \int_{B_{\rho}(0)}|\nabla (\hat{u}_{n}^{1}- u^1)^{2}|dx \to 0,  \ \ \ \
  \int_{B_{\rho}(0)}|\nabla (\hat{v}_{n}^{1}- v^1)^{2}|dx \to 0.
\end{equation}
Hence, we assert that $(u^{1},v^{1})\neq (0,0)$ due to \eqref{guo-add-new}.

Furthermore, we point out that $\sigma_{n}\to 0$ as $n\to \infty$.
If not,  we may suppose that $\sigma_n \to \sigma^*>0$ as $n \to +\infty$. Let
$$
(\tilde{u}_{n}^{1}(x), \tilde{v}_{n}^{1}(x)):=(u_{n}^{1}(x+y_n), v_{n}^{1}(x+y_n)).
$$
Taking into account \eqref{h3.3}, \eqref{mam} and $(\hat u_n,\hat v_n)\to (u^1,   v^1)\neq(0,0)$ in $L^2_{\lloc}(\R^4)\times L^2_{\lloc}(\R^4)$, it is readily seen that $|y_n|\to\infty$, as $n\to\infty$.
Then, $\{ (\tilde{u}_{n}^{1}, \tilde{v}_{n}^{1}) \}$ is a Palais-Smale sequence of $I_{\lambda,\infty}$ and also bounded in $H_{\lambda}$. We may assume that $ (\tilde{u}_{n}^{1}(x), \tilde{v}_{n}^{1}(x)) \rightharpoonup ( \tilde{u}^{1}, \tilde{v}^{1} )$ in $H_{\lambda}$, up to a subsequence. Moreover, we can find that $( \tilde{u}^{1}, \tilde{v}^{1} )$ solves the system (2.1).
Then by the nonexistence result, in Lemma 2.5, we assert that $(\tilde{u}^{1},\tilde{v}^{1})=(0,0)$, which implies that $(\tilde{u}_{n}^{1},  \tilde{v}_{n}^{1})\rightharpoonup (0,0)$ in $H_{\lambda}$. Thus, $\tilde{u}_{n}^{1}\to 0,  \tilde{v}_{n}^{1}\to 0$  in $L_{loc}^2(\R^4)$.
On the other hand,  using again the fact that $(\hat{u}_{n}^{1}, \hat{v}_{n}^{1})\rightharpoonup (u^{1},v^{1})\neq (0,0)$ in $H_0$ and a.e. on $\R^4 \times \R^4 $,  there exists a $r^*>0$ such that
$$
\int_{B_{r^*}(0)}(| u^{1}|^2+|v^{1}|^2)dx>0.
$$
Moreover, we get
\begin{align*}
\lim_{n\to\infty}\int_{B_{r^*\sigma^*}(0)}(| \tilde{u}_{n}^{1}|^2+| \tilde{v}_{n}^{1}|^2)dx&=\lim_{n\to\infty}\int_{B_{r^*\sigma^*}(0)}\sigma_n^{-2}\cdot(| \hat{u}_{n}^{1}(\sigma_n^{-1}x)|^2+| \hat{v}_{n}^{1}(\sigma_n^{-1}x)|^2)dx\\
&=\lim_{n\to\infty}\int_{B_{\frac{r^*\sigma^*}{\sigma_n}}(0)}\sigma_n^{2}\cdot(| \hat{u}_{n}^{1}(y)|^2+| \hat{v}_{n}^{1}(y)|^2)dy\\
&=(\sigma^*)^2\int_{B_{r^*}(0)}(| u^{1}(y)|^2+|v^{1}(y)|^2)dy>0,
\end{align*}
 which contradicts with the fact that $\tilde{u}_{n}^{1}\to 0,  \tilde{v}_{n}^{1}\to 0$ in $L_{loc}^2(\R^4)$.
 Hence, we have proved that $\sigma_{n}\to 0$ as $n\to \infty$.

In what follows,  we want to prove that  $(u^{1},v^{1})$ is a nonzero solution of the limit problem \eqref{S2}.
Indeed, for arbitrary $\varphi_1,\varphi_2 \in C_{0}^{\infty}(\R^4)$, from $\sigma_n\to 0$ it follows
$$
\langle I^{\prime}_{\infty}(\hat{u}_{n}^{1}, \hat{v}_{n}^{1}),(\varphi_{1},\varphi_{2})\rangle=\langle I^{\prime}_{\lambda,\infty}(\tilde{u}_{n}^{1}, \tilde{v}_{n}^{1}),(\hat{\varphi}_{1},\hat{\varphi}_{2})\rangle+o_{n}(1)=o_{n}(1),
$$
where $\hat{\varphi}_{j}=\sigma_n^{-1}\varphi_{j}(\frac{x}{\sigma_n}), j=1,2$. Then, $(u^1, v^1)$  solves the system \eqref{S2}.

\medskip

In the following, we set
\begin{equation}
\label{S1}
\big(u_{n}^{2} (x), v_{n}^{2}(x)\big)\! :=(u_{n}^{1}(x),v_{n}^{1}(x) )- \left(\varphi\left(\frac{x-{{y}}_n}{\sigma_{n}^{{1}/{2}}}\right)\sigma_{n}^{-1}u^{1}\left(\frac{x-{{y}}_n}{\sigma_n}\right),
   \varphi  \left(\frac{x-{{y}}_n}{\sigma_{n}^{ {1}/{2}}}\right)\sigma_{n}^{-1}v^{1}\left(\frac{x-{{y}}_n}{\sigma_n}\right)  \right),
\end{equation}
  where $\varphi(x)$  is  a cut-off function satisfying $\varphi(x)=1$ if $|x|\in (0,1]$, $\varphi(x)=0$ if $|x|>2$.
 We claim that $\{(u_{n}^{2}(x), v_{n}^{2}(x))\}$ is a Palais-Smale sequence for $I_{\lambda,\infty}$ in $H_{\lambda}$ and
\begin{equation}\label{e-3.13}
\big(u_{n}^{2}(x), v_{n}^{2}(x)\big)\rightharpoonup (0,0) \ \ \text{in} \ \ H_{\lambda}.
\end{equation}
Since $(u_{n}^{1},v_{n}^{1} )\rightharpoonup (0,0)$ in $H_{\lambda}$,  in order to prove \eqref{e-3.13} it is sufficient to verify
 \begin{equation}\label{e-3.10}
\left(\varphi\left(\frac{x-{{y}}_n}{\sigma_{n}^{{1}/{2}}}\right)\sigma_{n}^{-1}u^{1}\left(\frac{x-{{y}}_n}{\sigma_n}\right),  \
  \varphi \left(\frac{x-{{y}}_n}{\sigma_{n}^{ {1}/{2}}}\right)\sigma_{n}^{-1}v^{1}\left(\frac{x-{{y}}_n}{\sigma_n}\right)  \right) \rightharpoonup (0,0) \ \  \text{in} \ \  H_{\lambda}.
\end{equation}
Taking into account that   $\varphi$  is a cut-off function, we obtain
\begin{align}\label{e-3.11}
& \ \  \left\|
\left(\varphi\left(\frac{x-{{y}}_n}{\sigma_{n}^{{1}/{2}}}\right)\sigma_{n}^{-1}u^{1}\left(\frac{x-{{y}}_n}{\sigma_n}\right),  \
  \varphi\left(\frac{x-{{y}}_n}{\sigma_{n}^{ {1}/{2}}}\right)\sigma_{n}^{-1}v^{1}\left(\frac{x-{{y}}_n}{\sigma_n}\right)  \right)
\right\|_{L^{2}\times L^{2}}^{2}  \nonumber \\
& = \sigma_{n}^{2}\int_{\R^4} \varphi^{2}(\sigma_{n}^{\frac{1}{2}}x) |u^{1}(x)|^{2}dx+\sigma_{n}^{2}\int_{\R^4} \varphi^{2}(\sigma_{n}^{\frac{1}{2}}x) |v^{1}(x)|^{2}dx  \nonumber \\
&\leq \sigma_{n}^{2}\int_{|x|\leq 2\sigma_{n}^{-\frac{1}{2}}} |u^{1}(x)|^{2}dx+\sigma_{n}^{2}\int_{|x|\leq 2\sigma_{n}^{-\frac{1}{2}}}  |v^{1}(x)|^{2}dx \\
&\leq \sigma_{n}^{2} \Big( \int_{|x|\leq 2\sigma_{n}^{-\frac{1}{2}}}|u^1|^4dx\Big)^{\frac{1}{2}}\Big( \int_{|x|\leq 2\sigma_{n}^{-\frac{1}{2}}}1dx\Big)^{\frac{1}{2}}
+ \sigma_{n}^{2} \Big( \int_{|x|\leq 2\sigma_{n}^{-\frac{1}{2}}}|v^1|^4dx\Big)^{\frac{1}{2}}\Big( \int_{|x|\leq 2\sigma_{n}^{-\frac{1}{2}}}1dx\Big)^{\frac{1}{2}}  \nonumber \\
& \leq  C \sigma_n  \Big( \int_{|x|\leq 2\sigma_{n}^{-\frac{1}{2}}}|u^1|^4dx\Big)^{\frac{1}{2}}+C \sigma_n  \Big( \int_{|x|\leq 2\sigma_{n}^{-\frac{1}{2}}}|v^1|^4dx\Big)^{\frac{1}{2}}\to 0, \ \ \text{as} \,\, \sigma_n\to 0.  \nonumber
\end{align}
Meanwhile, a direct calculation shows that
\begin{align}
\label{e-3.12}
 & \ \ \left\| \varphi\left(\frac{x-{{y}}_n}{\sigma_{n}^{{1}/{2}}}\right)\sigma_{n}^{-1}u^{1}\left(\frac{x-{{y}}_n}{\sigma_n}\right) -\sigma_{n}^{-1}u^{1}\left(\frac{x-{{y}}_n}{\sigma_n}\right)\right\|_{D^{1,2}}^2 \nonumber \\
  & = \| \varphi(\sigma_{n}^{\frac{1}{2}}x )u^{1}(x)-u^{1}(x) \|_{D^{1,2}}^2 \nonumber \\
  & \leq 2 \int_{\R^4} \big( 1- \varphi(\sigma_{n}^{\frac{1}{2}}x ) \big)^2|\nabla u^1|^{2}dx + 2 \int_{\R^4} [\nabla\big( 1- \varphi(\sigma_{n}^{\frac{1}{2}}x ) \big)]^2| u^1|^{2}dx  \\
  & \leq 2 \int_{|x|\geq \sigma_{n}^{-\frac{1}{2}}} |\nabla u^1|^{2}dx + 2   \sigma_{n}C\int_{\sigma_{n}^{-\frac{1}{2}}\leq |x|\leq 2 \sigma_{n}^{-\frac{1}{2}}} | u^1|^{2}dx  \nonumber\\
  & \leq o_{n}(1)+ C\Big(  \int_{\sigma_{n}^{-\frac{1}{2}}\leq |x|\leq 2 \sigma_{n}^{-\frac{1}{2}}} | u^1|^{4}dx  \Big)^{\frac{1}{2}} \to 0,   \nonumber
\end{align}
as $n\to \infty$.
Thus, we get $\varphi(\frac{x-{{y}}_n}{\sigma_{n}^{{1}/{2}}})\sigma_{n}^{-1}u^{1}(\frac{x-{{y}}_n}{\sigma_n})\rightharpoonup 0$ in $D^{1,2}(\R^4)$ because a direct computation shows that  $\sigma_{n}^{-1}u^{1}(\frac{x-{{y}}_n}{\sigma_n})\rightharpoonup 0$ in $D^{1,2}(\R^4)$.
Similarly, we see that $ \varphi(\frac{x-{{y}}_n}{\sigma_{n}^{ {1}/{2}}})\sigma_{n}^{-1}v^{1}(\frac{x-{{y}}_n}{\sigma_n})\rightharpoonup 0$ in $D^{1,2}(\R^4)$.
So,
\begin{equation}\label{add-e-3.13}
\left(\varphi\left(\frac{x-{{y}}_n}{\sigma_{n}^{{1}/{2}}}\right)\sigma_{n}^{-1}u^{1}\left(\frac{x-{{y}}_n}{\sigma_n}\right),  \
\varphi\left(\frac{x-{{y}}_n}{\sigma_{n}^{ {1}/{2}}}\right)\sigma_{n}^{-1}v^{1}\left(\frac{x-{{y}}_n}{\sigma_n}\right)  \right) \rightharpoonup (0,0) \,\, \text{in} \, \, H_0.
\end{equation}
Then  \eqref{e-3.10} follows from \eqref{e-3.11} and \eqref{add-e-3.13} and so \eqref{e-3.13} is proved.

Based on the properties \eqref{e-3.11}, \eqref{e-3.12} and \eqref{e-3.13} above, together  with  Lemma \ref{lm3.2},  we have
\begin{equation}\label{e-3.14}
  I_{\lambda, \infty}(u_{n}^{2},v_{n}^{2})=I_{\lambda, \infty}(u_{n}^{1},v_{n}^{1})-I_{\infty}(u^1,v^1)+o_n(1)
\end{equation}
and
\begin{align*}
  \|(u_n,v_n)\|_{H_\lambda}^2&=\|(u^0,v^0)\|_{H_\lambda}^{2}+\|(u_n^1,v_{n}^{1})\|_{H_\lambda}^{2}+o_{n}(1) \\
  &=\|(u^0,v^0)\|_{H_\lambda}^{2}+\|(\sigma_{n}^{-1}u^{1}(\frac{x-{{y}}_n}{\sigma_n}), \sigma_{n}^{-1}v^{1}(\frac{x-{{y}}_n}{\sigma_n})   \|_{H_0}^2+\|(u_n^2,v_{n}^{2})\|_{H_\lambda}^{2}+o_{n}(1).
\end{align*}
For arbitrarily chosen test function $( w,z)\in H_\lambda$  with $\|(w,z)\|_{H_\lambda}\leq 1$,  by Lemma \ref{lm3.2} again, we have
\begin{align*}
& \ \ \langle I'_{\lambda,\infty}(u_{n}^{2},v_{n}^{2}), (w,z) \rangle\\
&=\langle  I'_{\lambda,\infty}(u_{n}^{1},v_{n}^{1}), (w,z) \rangle-  \big\langle  I'_{\lambda,\infty}\big(\varphi(\frac{x-{{y}}_n}{\sigma_{n}^{\frac{1}{2}}})\sigma_{n}^{-1}u^{1}(\frac{x-{{y}}_n}{\sigma_n}), \,\,  \varphi(\frac{x-{{y}}_n}{\sigma_{n}^{\frac{1}{2}}})\sigma_{n}^{-1}v^{1}(\frac{x-{{y}}_n}{\sigma_n}) \big), (w,z) \big \rangle   \\
&\phantom{=}+o_n(1)\\
&=\langle  I'_{\lambda,\infty}(u_{n}^{1},v_{n}^{1}), (w,z) \rangle - \langle I'_{\infty}(u^1,v^1), \big(\sigma_nw(\sigma_nx+{{y}}_n), \sigma_nz(\sigma_nx+{{y}}_n)\big)     \rangle+o_{n}(1) \\
&=\langle  I'_{\lambda,\infty}(u_{n}^{1},v_{n}^{1}), (w,z) \rangle+o_{n}(1).
\end{align*}
As we known that $\{(u_{n}^{1},v_{n}^{1})\}$ is a Palais-Smale sequence of $I_{\lambda,\infty}$, then for $ n\to \infty$,
\begin{align*}
  \|  I'_{\lambda,\infty}(u_{n}^{2},v_{n}^{2}) \|_{(H_\lambda)^{-1}}& = \sup_{(w,z)\in H_\lambda,\, \|(w,z)\|_{H_\lambda}\leq 1} \langle I'_{\lambda,\infty}(u_{n}^{2},v_{n}^{2}), (w,z) \rangle    \\
  &= \|I'_{\lambda,\infty}(u_{n}^{1},v_{n}^{1})\|_{(H_\lambda)^{-1}}+o_{n}(1) \to 0.
\end{align*}
Finally, we have proved that $\{(u_{n}^{2},v_{n}^{2})\}$ is a Palais-Smale  sequence of $I_{\lambda,\infty}$. Moreover, by \eqref{e-3.14} we get
\begin{align*}
I(u_n,v_n)&=I_{\lambda,\infty}(u_{n}^{1},v_{n}^{1})+I(u^0,v^0)+o_{n}(1)=I_{\lambda, \infty}(u_{n}^{2},v_{n}^{2})+I_{\infty}(u^1,v^1)+I(u^0,v^0)+o_{n}(1).
\end{align*}
If $(u_{n}^{2},v_{n}^{2}) \to (0,0)$ in $H_\lambda$, we are done.
Otherwise,  we can iterate the procedure and at the step $\ell$ we get
\begin{align*}
  \|(u_n,v_n)\|_{H_\lambda}^{2} =\|(u^0,v^0)\|_{H_\lambda}^{2}
 & +\sum_{k=1}^{\ell}\|\Big((\sigma_{n}^{k})^{-1}u^k(\frac{\cdot-{{y}}_{n}^{k}}{r_{n}^{k}}),(\sigma_{n}^{k})^{-1}v^k(\frac{\cdot-{{y}}_{n}^{k}}{r_{n}^{k}})\Big)\|_{ H_0}^{2}  \\
&  +\|(u_n^{\ell+1},v_n^{\ell+1})\|^2_{H_\lambda}+o_n(1)
\end{align*}
and
\begin{equation*}
  I(u_n,v_n)=I(u^0,v^0)+\sum_{k=1}^{\ell}I_{\infty}(u^{k},v^k)+ I_{\lambda,\infty}(u_{n}^{\ell+1},  v_{n}^{\ell+1})+o_{n}(1),
\end{equation*}
where $(u^{k},v^{k})$ are non-zero solutions of system \eqref{S2}, for $k\in\{1,\ldots,  \ell  \}$.
Since $(u^{k},v^{k})\in\cN_\infty$, by Proposition \ref{PN} we get
$$
 \|(u_n,v_n)\|_{H_\lambda}^{2} \ge\|(u^0,v^0)\|_{H_\lambda}^{2}+\ell\,\cC+o_n(1).
$$
Taking into account that $ \|(u_n,v_n)\|_{H_\lambda}^{2}$ is bounded, we can conclude that the iteration must terminate at a finite index $\ell$, such that $\|(u_{n}^{\ell+1},v_{n}^{\ell+1})\|_{H_\lambda}\to 0$.

So
\begin{equation*}
  \|(u_n,v_n)\|_{H_\lambda}^{2}=\|(u^0,v^0)\|_{H_\lambda}^{2}+\sum_{k=1}^{\ell}\|\Big((\sigma_{n}^{k})^{-1}u^k(\frac{\cdot-{{y}}_{n}^{k}}{r_{n}^{k}}),(\sigma_{n}^{k})^{-1}v^k(\frac{\cdot-{{y}}_{n}^{k}}{r_{n}^{k}})\Big)\|_{ H_0}^{2}+o_{n}(1)
\end{equation*}
and
\begin{equation*}
I(u_n,v_n)=I(u^0,v^0)+\sum_{k=1}^{\ell}I_{\infty}(u^k,v^k)+o_{n}(1).
\end{equation*}
{\hfill$\blacksquare$\vspace{6pt}}

\begin{cor}\label{co3.4}
  If $\{(u_n,v_n)\}  \subset \mathcal{N}$ is a Palais-Smale sequence for the constrained functional $I|_{\mathcal{N}}$ at level $d\in \left(c_{\infty}, \min\left\{\frac{\mathcal{S}^{2}}{4\mu_1}, \frac{\mathcal{S}^{2}}{4\mu_2}, 2c_{\infty} \right\}\right)$, then $\{(u_n,v_n)\}$ is relatively compact.
\end{cor}

\begin{proof}
Let $\{(u_n,v_n)\}$ be a Palais-Smale sequence for $I|_{\mathcal{N}}$ at level $d$, that is, $\lim_{n\to \infty}I(u_n,v_n)= d$ and $\lim_{n\to \infty}(I|_{\mathcal{N}})'(u_n,v_n) = 0$. It is easy to prove that  $\lim_{n\to \infty}I'(u_n,v_n)=0$
and $\{(u_n,v_n)\}$ is bounded  in $H_{\lambda}$.
Let $(u_n,v_n)\rightharpoonup (u^0,v^0)$ in $H_{\lambda}$, up to a subsequence.
If $(u_n,v_n) \to (u^0,v^0)$ in $H_{\lambda}$, we are done.
If not, then it follows from Theorem \ref{th3.1} that there exist sequences  $\{(u^k,v^k)\}$, $\{\sigma_{n}^{k}\}$ and $\{{{y}}_{n}^{k}\}$, $k=1,2,\cdots,\ell$, such that
 \begin{equation*}
  \|(u_n,v_n)\|_{H_\lambda}^{2}=\|(u^0,v^0)\|_{H_\lambda}^{2}+\sum_{k=1}^{\ell}\|\Big((\sigma_{n}^{k})^{-1}u^k(\frac{\cdot-{{y}}_{n}^{k}}{\sigma_{n}^{k}}),\ (\sigma_{n}^{k})^{-1}v^k(\frac{\cdot-{{y}}_{n}^{k}}{\sigma_{n}^{k}})\Big)\|_{H_0}^{2}+o_{n}(1)
\end{equation*}
and
\begin{equation*}
I(u_n,v_n)=I(u^0,v^0)+\sum_{k=1}^{\ell}I_{\infty}(u^k,v^k)+o_{n}(1).
\end{equation*}
Note that  $d<2c_{\infty}$ implies $\ell=1$, that is
 \begin{equation*}
d=I(u^0,v^0)+I_{\infty}(u^1,v^1).
\end{equation*}

 If $(u^0,v^0)\neq (0,0)$, then from Lemma \ref{lm2.6} we get
 \begin{equation*}
   d=I(u^0,v^0)+I_{\infty}(u^1,v^1)> c+c_{\infty}=2c_{\infty},
 \end{equation*}
which contradicts the assumption $d\in \left(c_{\infty}, \min\left\{\frac{\mathcal{S}^{2}}{4\mu_1}, \frac{\mathcal{S}^{2}}{4\mu_2}, 2c_{\infty} \right\}\right)$.

Then, let $(u^0,v^0)=(0,0)$. Since $u^1$ and $v^1$ have definite sign, then by Lemma \ref{lm2.4} and the uniqueness of positive solutions for the scalar equation
\begin{equation*}
-\Delta u=\mu_{j}u^{3},  \, \,     x\in \R^4,   \, \, j=1,2,
\end{equation*}
we deduce that $(u^1,v^1)$ must be one of the following three  solutions, up to translations and dilations,
$$(\sqrt{k_1}U_{1,0},\sqrt{k_2}U_{1,0} ), \ \ \ \ \left(\frac{1}{\sqrt{\mu_1}}U_{1,0},0\right), \ \ \ \  \left(0, \frac{1}{\sqrt{\mu_2}}U_{1,0}\right).$$
Therefore,
$$
\text{either} \ \ d=c_{\infty}, \ \ \text{or} \ \ d=\frac{1}{4\mu_{1}}\|U_{1,0}\|_{D^{1,2}}^{2}=\frac{\mathcal{S}^2}{4\mu_1}, \ \ \text{or} \ \ d=\frac{1}{4\mu_{2}}\|U_{1,0}\|_{D^{1,2}}^{2}=\frac{\mathcal{S}^2}{4\mu_2},
$$
which also contradicts with the assumption $d\in \left(c_{\infty}, \min\left\{\frac{\mathcal{S}^{2}}{4\mu_1}, \frac{\mathcal{S}^{2}}{4\mu_2}, 2c_{\infty} \right\}\right)$.
So $\ell=0$ and $(u_n,v_n) \to (u^0,v^0)$ in $H_{\lambda}$.
\end{proof}

\begin{prop}
\label{CorMin}
Let $\{(u_n,v_n)\}$ in $H_\lambda$, with $\lambda\ge 0$, be such that
\beq
\label{mam3}
\lim_{n\to\infty}I(u_n,v_n) =c_\infty,
\eeq
then there exist $\{y_n\}$ in $\R^4$ and $\{\sigma_n\}$ in $(0,+\infty)$ such that
$$
(u_n,v_n)=(\sqrt{k_1}U_{\sigma_n,y_n},\sqrt{k_2} U_{\sigma_n,y_n})+o_n(1)
$$
where
$$
\sigma_n\to 0\ \mbox{ for } \ \lambda>0\quad\mbox{ and }\quad o_n(1)\to 0\   \mbox{ in }H_0.
$$
\end{prop}

\begin{proof}
If $\lambda=0$ the assertion is contained in \cite[Theorem 3.1]{LL}.
Then, we assume $\lambda>0$.

By the Ekeland variational principle, we can assume that $\{(u_n,v_n)\}$ is a Palais-Smale sequence.
Then, we can apply Theorem \ref{th3.1} and from \eqref{mam3} and Lemma \ref{lm2.4} we infer $(u^0,v^0)=(0,0)$, $\ell=1$, $\sigma_n:=\sigma_n^1\to 0$, and we can consider $(u_1,v_1)=(\sqrt{k_1}U_{1,0},\sqrt{k_2} U_{1,0})$.
As a consequence, we can observe that, in the notations of Theorem \ref{th3.1}, $(u_n^1,v_n^1)=(u_n,v_n)$ (see \eqref{h3.3}) and $(u_n^2,v_n^2)\to 0$ in $H_\lambda$ (see \eqref{S1}) because otherwise $\ell\ge 2$.
Then, by \eqref{S1} it turns out
$$ 
(u_n(x),v_n(x))=\left(\varphi\left(\frac{x-{{y}}_n}{\sigma_{n}^{{1}/{2}}}\right) \sqrt{k_1}U_{\sigma_n,y_n}(x),
\   \varphi \left(\frac{x-{{y}}_n}{\sigma_{n}^{ {1}/{2}}}\right)\sqrt{k_2}U_{\sigma_n,y_n}(x)\right)
+\Phi_n(x)
$$
where $\Phi_n\to 0$ in $H_\lambda$.
Then, the statement follows from
$$
 \left(\varphi\left(\frac{\cdot-{{y}}_n}{\sigma_{n}^{{1}/{2}}}\right) \sqrt{k_1}U_{\sigma_n,y_n},
\   \varphi  \left(\frac{\cdot-{{y}}_n}{\sigma_{n}^{ {1}/{2}}}\right)\sqrt{k_2}U_{\sigma_n,y_n}\right)-(\sqrt{k_1}U_{\sigma_n,y_n},\sqrt{k_2} U_{\sigma_n,y_n})\to 0 \qquad\mbox{ in }H_0.
$$
\end{proof}

\section{Main tools and basic estimates}\label{s4}

\quad The main goal of this section is to introduce some tools and  establish some basic estimates in order to prove the  existence and multiplicty of  nontrivial solutions of system \eqref{S-4}.
  For $(u,v)\in  H_{0}$  such that  $(u^+,v^+)\neq (0,0)$ we first define a barycenter map by
\begin{equation*}
\xi(u,v)=\frac{\int_{\R^4}\frac{x}{1+|x|}\big(\mu_1(u^+)^4+2\beta (u^+)^2(v^+)^2+\mu_2 (v^+)^4\big)dx  }{\int_{\R^4}\big(\mu_1(u^+)^4+2\beta (u^+)^2(v^+)^2+\mu_2(v^+)^4\big)dx  }
\end{equation*}
and set
\begin{equation*}
  \gamma(u,v)=\frac{\int_{\R^4}|\frac{x}{1+|x|}-\xi(u,v)|\big(\mu_1(u^+)^4+2\beta (u^+)^2(v^+)^2+\mu_2 (v^+)^4\big)dx  }{\int_{\R^4}\big(\mu_1(u^+)^4+2\beta (u^+)^2(v^+)^2+\mu_2(v^+)^4\big)dx  }
\end{equation*}
to estimate the concentration of $(u,v)$ around its barycenter.
It is not difficult to find that $\xi,\gamma$ are continuous with respect to the $H_{0}$-norm, and for all $t >0$ and $(u,v)\in H_{0}$  such that $(u^+,v^+)\neq (0,0)$
\begin{equation}\label{e-4.1}
\xi(tu,tv)=\xi(u,v),  \ \ \gamma(tu,tv)=\gamma(u,v).
\end{equation}

Moreover,  we introduce  the notation $I_{0}$ to denote by the functional $I$ with $\lambda =0$, that is,
\begin{align*}
  I_0(u,v)&=\frac{1}{2}\int_{\R^4}\Big(|\nabla u|^{2}+|\nabla v|^{2} \Big)dx
  + \frac{1}{2}\int_{\R^4}\Big(V_{1}(x)u^2+V_2(x)v^2\Big)dx \\
  & \qquad -\frac{1}{4}\int_{\R^4}\big(\mu_1(u^+)^4+2\beta (u^+)^2(v^+)^2+\mu_2 (v^+)^4     \big)dx.
\end{align*}
According to this, we define the  Nehari manifold associated with $I_0$ by
\begin{equation}
\label{N0}
 \mathcal{ N}_{0}:=\{(u,v)\in H_{0} \setminus  \{  (0,0)  \} : \langle I'_{0}(u,v),(u,v)    \rangle=0  \}.
\end{equation}

 \begin{lem}\label{lm4.1}
Let
\begin{equation*}
c^*=\inf\left\{I(u,v):(u,v)\in \mathcal{N},  \, \, \xi(u,v)=0, \, \, \gamma(u,v)=\frac{1}{2}  \right\},
\end{equation*}
\begin{equation*}
c^*_0=\inf\left\{I_0(u,v):(u,v)\in \mathcal{N}_0,  \, \, \xi(u,v)=0, \, \, \gamma(u,v)=\frac{1}{2}  \right\}.
\end{equation*}

Then, the following inequalities hold
\beq
\label{c*}
c_{\infty}< c^*_0\le c^*.
\eeq
\end{lem}
\begin{proof}
First, let us prove that from $\lambda>0$ there follows $c^*\ge c^*_0$.
Indeed, let $(u,v)\in \mathcal{N}$ and let $\bar t>0$ be such that $(\bar t u,\bar t v)\in \mathcal{N}_0$.
Arguing as in Lemma \ref{lm2.2}, we can see that $\bar t<1$, then
\begin{align*}
I(u,v)&=\frac 14  \Big[ \int_{\R^4}|\nabla u_n|^{2}dx+\int_{\R^4}|\nabla v_n|^{2}dx+\int_{\R^4}(V_{1}(x)+\lambda) |u_n|^{2}dx+ \int_{\R^4}(V_{2}(x)+\lambda) |v_n|^{2}dx \Big]\\
   &\ge \frac {\bar t^2}{4}  \Big[ \int_{\R^4}|\nabla u_n|^{2}dx+\int_{\R^4}|\nabla v_n|^{2}dx+\int_{\R^4}(V_{1}(x)+\lambda) |u_n|^{2}dx+ \int_{\R^4}(V_{2}(x)+\lambda) |v_n|^{2}dx \Big]\\
 &\ge \frac {\bar t^2}{4}  \Big[ \int_{\R^4}|\nabla u_n|^{2}dx+\int_{\R^4}|\nabla v_n|^{2}dx+\int_{\R^4}V_{1}(x) |u_n|^{2}dx+ \int_{\R^4}V_{2}(x) |v_n|^{2}dx \Big]\\
&=I_0(\bar t u,\bar t v).
\end{align*}
Hence, taking into account \eqref{e-4.1}, we get $c^*\ge c^*_0$.

\medskip

Now, let us prove $c^*_0>c_\infty$.
From Lemma \ref{lm2.6}, $c^{*}_0\geq c= c_{\infty}$ follows.
Suppose by contradiction that $c^*_0=c_{\infty}$.
Then, there exist  a sequence of $\{(u_n,v_n)\}\subset \mathcal{N}_0$ such that
\begin{equation}
\label{e-4.3}
  \xi(u_n,v_n)=0,\quad \gamma(u_n,v_n)= \frac{1}{2}\qquad\forall n\in\N,
\end{equation}
and
\begin{equation*}
  \lim_{n\to\infty}I_0(u_n,v_n)= c_{\infty}.
\end{equation*}
 It follows from Proposition \ref{CorMin} that there exist   $\sigma_n>0$  and $y_n\in\R^4$ such that,
 \begin{equation}\label{e-4.4}
  (u_n,v_n)=(\sqrt{k_1}U_{\sigma_n,y_n}, \sqrt{k_2}U_{\sigma_n,y_n})+(\varphi_n,\psi_n),
\end{equation}
where $(\varphi_n,\psi_n) \to (0,0)$  in $H_{0}$ as $n\to \infty$.

Now, we claim that, up to a subsequence,
\begin{equation}\label{e-4.5}
(a) \lim_{n\to \infty}\sigma_n= \sigma_0>0, \ \ \ \  \ \ \ \ \ \ \ \  \ \ \ \ (b) \lim_{n\to\infty}y_n=y_0\in \R^4.
\end{equation}
In order to prove \eqref{e-4.5}$(a)$, we start  to show that $\{\sigma_n\}$ is bounded.
Assume, by contradiction, that there exists a subsequence $\{\sigma_{n_j}\}$ of $\{\sigma_n\}$, still denoted by $\{\sigma_n\}$, such that $\lim_{n\to\infty}\sigma_n=\infty$.
Then,  a direct calculation shows that for every fixed $r>0$
\begin{equation*}
  \lim_{n\to \infty}\int_{B_{r}(0)}\big(\mu_1(u_n^+)^4+2\beta (u_n^+)^2(v_n^+)^2+\mu_2 (v_{n}^+)^4\big)dx\le \frac{C\, r^4}{\sigma_n^4}=0.
\end{equation*}
Hence, taking also into account that $\lim_{n\to\infty}\xi(u_n,v_n)=0$ by  \eqref{e-4.3}, we get
\begin{align*}
  \frac{1}{2}&=\lim_{n\to \infty}\gamma (u_n,v_n) \\
  &=\lim_{n\to \infty}\frac{\int_{\R^4}\frac{|x|}{1+|x|}\big(\mu_1(u_n^+)^4+2\beta (u_n^+)^2(v_n^+)^2+\mu_2 (v_n^+)^4\big)dx   }{\int_{\R^4}\big(\mu_1(u_n^+)^4+2\beta (u_n^+)^2(v_n^+)^2+\mu_2(v_n^+)^4\big)dx   }\\
  &=\lim_{n\to \infty}\frac{\int_{\R^4\setminus B_{r}(0)}\frac{|x|}{1+|x|}\big(\mu_1(u_n^+)^4+2\beta (u_n^+)^2(v_n^+)^2+\mu_2 (v_n^+)^4\big)dx +o_{n}(1) }{\int_{\R^4 \setminus B_{r}(0)}\big(\mu_1(u_n^+)^4+2\beta (u_n^+)^2(v_n^+)^2+\mu_2(v_n^+)^4\big)dx +o_{n}(1) } \\
  & \geq \frac{r}{1+r}.
\end{align*}
This inequality  is impossible for  $r>1$ and thus  $\{\sigma_n\}$ must be bounded.

Up to a subsequence, we suppose that $\lim_{n\to \infty}\sigma_n \to \sigma_0$ with   $\sigma_0 \geq 0$.

\smallskip

{ \em\underline  {Claim (*)}}: $\sigma_0=0$ cannot occur.\ If $\sigma_0=0$, then for each $r>0$, we have
\begin{equation}
\label{en1}
  \lim_{n\to\infty}\int_{\R^4\setminus B_{r}(y_n)}\big(\mu_1(u_n^+)^4+2\beta (u_n^+)^2(v_n^+)^2+\mu_2(v_n^+)^4\big)dx =(k_1+k_2)\lim_{n\to \infty} \int_{\R^4 \setminus B_{r}(0)}|U_{\sigma_n,0}|^{4}dx=0.
\end{equation}
Since $\lim_{n\to \infty}\xi(u_n,v_n)= 0$,   we get for each $r>0$,
\begin{equation}
\label{en2}
\begin{split}
\frac{|y_n|}{1+|y_n|}&= \Big| \frac{y_n}{1+|y_n|}-\xi(u_n,v_n) \Big|+o_{n}(1) \\
&\le\frac{\int_{\R^4}\Big|\frac{y_n}{1+|y_n|}-\frac{x}{1+|x|}  \Big| \big(\mu_1(u_n^+)^4+2\beta (u_n^+)^2(v_n^+)^2+\mu_2 (v_n^+)^4\big)dx   }{\int_{\R^4}\big(\mu_1(u_n^+)^4+2\beta (u_n^+)^2(v_n^+)^2+\mu_2(v_n^+)^4\big)dx   }+o_{n}(1) \\
&=\frac{\int_{B_{r}(y_n)}\Big| \frac{y_n}{1+|y_n|}-\frac{x}{1+|x|} \Big| \big(\mu_1(u_n^+)^4+2\beta (u_n^+)^2(v_n^+)^2+\mu_2 (v_n^+)^4\big)dx   }{\int_{B_{r}(y_n)}\big(\mu_1(u_n^+)^4+2\beta (u_n^+)^2(v_n^+)^2+\mu_2(v_n^+)^4\big)dx   }+o_{n}(1)  \\
&\leq 2r+o_{n}(1).
\end{split}
\eeq
Then $|y_n|\to 0$ as $n\to \infty$. Moreover, by the same calculation performed above we get
\begin{equation}
\label{en3}
\begin{split}
  \lim_{n\to \infty}\gamma(u_n,v_n)&=\lim_{n\to\infty}\frac{\int_{\R^4}\Big|\frac{x}{1+|x|}- \xi(u_n,v_n) \Big| \big(\mu_1(u_n^+)^4+2\beta (u_n^+)^2(v_n^+)^2+\mu_2 (v_n^+)^4\big)dx   }{\int_{\R^4} \big(\mu_1(u_n^+)^4+2\beta (u_n^+)^2(v_n^+)^2+\mu_2(v_n^+)^4 \big)dx   } \\
  &=\lim_{n\to\infty}\frac{\int_{\R^4}\Big|\frac{x}{1+|x|}- \frac{y_n}{1+|y_n|} \Big| \big(\mu_1(u_n^+)^4+2\beta (u_n^+)^2(v_n^+)^2+\mu_2 (v_n^+)^4\big)dx   }{\int_{\R^4} \big(\mu_1(u_n^+)^4+2\beta (u_n^+)^2(v_n^+)^2+\mu_2(v_n^+)^4 \big)dx   }=0,
\end{split}
\eeq
that is impossible by   \eqref{e-4.3}.
Thus, {\em Claim (*)} is verified  and we have proved that \eqref{e-4.5}$(a)$ holds true.

\smallskip

Next, we want to show that $\{y_n\} \subset \R^4$ is bounded. We argue by contradiction and suppose that $\lim_{n\to \infty}|y_n|=\infty$.
For each $\varepsilon>0$, we can find $r_0=r_{0}(\varepsilon)>0$ such that for all $n\in \N$,
\begin{equation}\label{e-4.8}
  \int_{\R^4 \setminus B_{r_0}(y_n)}|U_{\sigma_0,y_n}|^{4}dx=\int_{\R^4 \setminus B_{r_0}(0)}|U_{\sigma_0,0}|^{4}dx<\varepsilon.
\end{equation}
 For such $r_0$, there   exists  $N^*\in \N$  such that for all $n\geq N^*$  and $x\in B_{r_0}(y_n)$,
\begin{equation}\label{e-4.7}
  \left|\frac{x}{1+|x|}-\frac{y_n}{1+|y_n|} \right|<\varepsilon.
\end{equation}
Then,  by  \eqref{e-4.8} and \eqref{e-4.7}, we prove for  $n$ large enough,
\begin{align*}
  \left|\xi(u_n,v_n)-\frac{y_n}{1+|y_n|}   \right|
  &\leq \frac{\int_{\R^4}\Big|\frac{x}{1+|x|}- \frac{y_n}{1+|y_n|} \Big|\big (\mu_1|u_n|^4+2\beta (u_n^+)^2(v_n^+)^2+\mu_2 (v_n^+)^4\big)dx   }{\int_{\R^4}\big(\mu_1(u_n^+)^4+2\beta (u_n^+)^2(v_n^+)^2+\mu_2(v_n^+)^4\big)dx   } \\
  &=\frac{ (k_1+k_2)\int_{\R^4}\Big|\frac{x}{1+|x|}- \frac{y_n}{1+|y_n|} \Big| |U_{\sigma_0, y_n}|^{4}dx +o_{n}(1)  }{(k_1+k_2)\int_{\R^4}|U_{\sigma_0,y_n}|^{4}dx+o_{n}(1)   } \\
  &\leq \frac{ (k_1+k_2) \big[\varepsilon \int_{B_{r_0}(y_n)} |U_{\sigma_0, y_n}|^{4}dx+ 2\int_{\R^4  \setminus B_{r_0}(y_n)}|U_{\sigma_0,y_n}|^{4}dx \big] +o_{n}(1)}{(k_1+k_2)\mathcal{S}^2+o_{n}(1)   } \\
  & \leq \frac{(k_1+k_2)(\mathcal{S}^2+2)\varepsilon+o_{n}(1)}{(k_1+k_2)\mathcal{S}^2+o_{n}(1)}\leq C\varepsilon,
\end{align*}
where the positive constant  $C$ is independent of $y_n$ and $r_0$.
Therefore, from the inequality above together with the assumption $\lim_{n\to \infty}|y_n|=\infty$,  we infer
\begin{equation*}
  \lim_{n\to \infty}|\xi(u_n,v_n)| =1,
\end{equation*}
which contradicts \eqref{e-4.3}.
Therefore $\{y_n\}\subset \R^4$ is bounded and then \eqref{e-4.5}$(b)$ is true.

Based on  \eqref{e-4.4} and \eqref{e-4.5}, we deduce that
%
%
\begin{align*}
  4c_{\infty}&=\lim_{n\to \infty} \Big[ \int_{\R^4}|\nabla u_n|^{2}dx+\int_{\R^4}|\nabla v_n|^{2}dx+\int_{\R^4}V_{1}(x) |u_n|^{2}dx+ \int_{\R^4}V_{2}(x) |v_n|^{2}dx \Big]\\
  &\geq (k_1+k_2) \int_{\R^4}|\nabla U_{\sigma_0,y_0}|^{2}dx+  \int_{\R^4}\big(k_1V_1(x)+k_2 V_2(x)\big)|U_{\sigma_0,y_0}|^{2}dx \\
  &>(k_1+k_2)\mathcal{S}^2=4c_{\infty},
\end{align*}
which gives a contradiction and thus  the  proof is completed.
\end{proof}

\begin{lem}\label{L4.2}
Suppose that $\lambda >0$ and define
\begin{equation*}
c^{**}=\inf\{I(u,v):(u,v)\in \mathcal{N},  \, \, \xi(u,v)=0, \, \, \gamma(u,v)\geq\frac{1}{2}  \}.
\end{equation*}
Then
\begin{equation}\label{e-4.9}
c^{**}>c_{\infty}
\end{equation}
\end{lem}
 \begin{proof}
 Obviously, $c^{**}\geq c_{\infty}$.
If the equality holds, then there exists a sequence $\{(u_n,v_n)\}\subset \mathcal{N}$ such that
\begin{equation}\label{e-4.10}
 \xi(u_n,v_n)= 0, \quad  \gamma(u_n,v_n)\geq \frac{1}{2}, \qquad\forall n\in\N
\end{equation}
and
\begin{equation}\label{e-4.11}
  \lim_{n\to\infty} I(u_n,v_n)= c_{\infty}.
\end{equation}
Then, by Proposition \ref{CorMin} we have
\begin{equation}
\label{222}
(u_n,v_n)=(\sqrt{k_1}U_{\sigma_n,y_n}, \sqrt{k_2}U_{\sigma_n,y_n})+(\tilde{\varphi}_n,\tilde{\psi}_n),
\end{equation}
where $\sigma_{n} >0$,    $y_n\in \R^N$, $(\tilde{\varphi}_n,\tilde{\psi}_n) \to 0$ in $H_{0} $ and moreover $\sigma_n\to 0$.

 Working again as the {\em Claim (*)} in the proof of Lemma \ref{lm4.1}, we see that $\sigma_n\to 0$ is in contradiction with $\gamma(u_n,v_n)\geq \frac 12$, so the proof is completed.

%
%
%
%

\end{proof}

 By assumption $(A_3)$, we can choose  $a\in (0,1)$  such that
\beq
\label{e-4.15}
\begin{split}
 0< & \frac{\beta-\mu_2}{2\beta-\mu_1-\mu_2}\|V_1\|_{L^2}+\frac{\beta-\mu_1}{2\beta-\mu_1-\mu_2}\|V_2\|_{L^2} \\
   &=  2^{-\frac{1-a}{2}}  \min \left\{\sqrt{\frac{\beta^2-\mu_1\mu_2}{\mu_1(2\beta-\mu_1-\mu_2)}},
   \sqrt{\frac{\beta^2-\mu_1\mu_2}{\mu_2(2\beta-\mu_1-\mu_2)}}, \sqrt{2} \right\}  \mathcal{S}-\mathcal{S}.
\end{split}
\eeq
Note that
$$
2^{1-a}c_\infty\leq \min\Big\{\frac{\mathcal{S}^{2}}{4\mu_1}, \frac{\mathcal{S}^2}{4\mu_2}, 2c_\infty\Big\}.
$$

Then, we fix a constant $\bar{c}$ such that
\begin{equation}
\label{e-4.14}
c_{\infty}<\bar{c}<   \min\left\{\frac{c^*_0+c_{\infty}}{2},  2^{1-a}c_{\infty} \right\}.
\end{equation}

In what follows, we fix a nonnegative radial function $\vartheta(x)\in C_{0}^{\infty}(\R^4)$, with the following properties: $\supp \vartheta \subset B_{1}(0)$, $\vartheta$ is non-increasing with respect to $r=|x|$, $\|\vartheta\|_{D^{1,2}}^2=\|\vartheta\|_{L^{4}}^4>\mathcal{S}^{2}$, and
\begin{equation}\label{e-4.15-1}
  I_{\infty}(\sqrt{k_1}\vartheta, \sqrt{k_2}\vartheta)=\Sigma \in (c_{\infty}, \bar{c}).
\end{equation}
The existence of such an approximation of the Aubin-Talenti function is standard and can be seen taking into account \eqref{BN}, for $\mu_j=1$ (see also \cite{BN}).
For each $\delta>0$ and $y\in \R^4$, we set
\begin{equation*}
\vartheta_{\delta,y}=
  \begin{cases}
   \delta^{-1}\vartheta(\frac{x-y}{\delta}), &x\in B_{\delta}(y); \\
   0, \ \ &x\notin B_{\delta}(y).
  \end{cases}
\end{equation*}
A direct calculation shows that
\begin{equation}
  \|\vartheta\|_{L^{4}}=\|\vartheta\|_{L^{4}(B_1(0))}=\|\vartheta_{\delta,y}\|_{L^{4}(B_{\delta}(y))}=\|\vartheta_{\delta,y}\|_{L^{4}}.
\end{equation}

Next lemmas \ref{lm4.4} and \ref{lm4.5} state the behavior of the test functions  $\vartheta_{\delta,y}$ with respect to the potential term and with respect to the barycenter and concentration maps.
Similar results can be found in \cite{CM,LL}, anyway we repeat their proofs here for the sake of completeness.

\begin{lem}\label{lm4.4}
Suppose that $(A_1)$-$(A_2)$ hold, then

\smallskip

$(a)$~$\displaystyle\limsup_{\delta\to 0}\left\{\int_{\R^4} V_{j}(x)|\vartheta_{\delta,y}|^{2}dx: y\in \R^4\right\}=0$, $j=1,2$;

\smallskip

$(b)$~$\displaystyle\limsup_{\delta \to  \infty} \left\{\int_{\R^4} V_{j}(x)|\vartheta_{\delta,y}|^{2}dx:y\in \R^4  \right\}=0 $, $j=1,2$;

\smallskip

$(c)$~$\displaystyle\limsup_{r\to  \infty} \left\{\int_{\R^4} V_{j}(x)|\vartheta_{\delta,y}|^{2}dx: |y|=r, y\in \R^4, \delta>0 \right\}=0$,\  \ $j=1,2$.
\end{lem}

 \begin{proof}
Let $y\in \R^4$ be chosen  arbitrarily. For all $\delta>0$,  by H\"{o}lder inequality we have
\begin{align*}
\int_{\R^4}V_j(x)|\vartheta_{\delta,y}(x)|^{2}dx&=\int_{B_{\delta}(y)}V_j(x)|\vartheta_{\delta,y}(x)|^{2}dx \nonumber \\
&\leq \|V_j\|_{L^{2}(B_{\delta}(y))}\|\vartheta\|_{L^{4}(B_{1}(0))}^{2} \leq C\|V_j\|_{L^{2}(B_{\delta}(y))},
\end{align*}
where the positive constant $C$ is  independent of $\delta$. Therefore,
\begin{equation*}
  \sup_{y\in \R^4} \int_{\R^4}V_j(x)|\vartheta_{\delta,y}(x)|^{2}dx
  \leq C \sup \big\{ \|V_j\|_{L^{2}(B_{\delta}(y))}: y\in \R^4 \big\},
\end{equation*}
and thus $(a)$ follows from
 \begin{equation*}
   \lim_{\delta\to0} \|V_j\|_{L^{2}(B_{\delta}(y))}=0  \,\, \text{uniformly in } \, y\in \R^4.
 \end{equation*}

To prove $(b)$, we observe that for each $y\in \R^4$, $\rho , \delta>0$, by H\"{o}lder inequality again, we get
\begin{align*}
  \int_{\R^4}V_j(x)|\vartheta_{\delta,y}|^{2}dx&=\int_{B_{\rho}(0)}V_j(x)|\vartheta_{\delta,y}|^{2}dx
+\int_{\R^4 \setminus B_{\rho}(0)}V_j(x)|\vartheta_{\delta,y}|^{2}dx \\
  &\leq \|V_j\|_{L^{2}(B_{\rho}(0))}\|\vartheta_{\delta,y}\|_{L^{4}(B_{\rho}(0))}^{2}
+\|V_j\|_{L^{2}(\R^4 \setminus B_{\rho}(0))}\|\vartheta_{\delta,y}\|_{L^{4}(\R^4 \setminus B_{\rho}(0))}^{2} \\
& \leq  \|V_j\|_{L^{2}}  \|\vartheta_{\delta,0}\|_{L^{4}(B_{\rho}(0))}^{2}
+\|\vartheta\|_{L^4}^2\|V_j\|_{L^{2}(\R^4 \setminus B_{\rho}(0))}.
\end{align*}
Since
\begin{equation*}
  \lim_{\delta \to  \infty} \|\vartheta_{\delta,0}\|_{L^{4}(B_{\rho}(0))}=0,
\end{equation*}
we get
\begin{equation*}
  \lim_{\delta \to  \infty}\sup_{y\in \R^4}\int_{\R^N}V_j(x)|\vartheta_{\delta,y}|^{2}dx\leq C\|V_j\|_{L^{2}(\R^N\backslash B_{\rho}(0))}.
\end{equation*}
So,   letting $\rho \to \infty$ in the above inequality, we prove that $(b)$ is true.

To verify $(c)$, we argue by contradiction and assume that there exist sequences of $\{y_n\}$ and $\{\delta_n\}$ with $\delta_{n}\in \R^+\setminus \{0\}$, such that
\begin{equation}\label{e-4.18}
  |y_n|\to  \infty
\end{equation}
and
\begin{equation}\label{e-4.19}
  \lim_{n\to \infty}\int_{\R^4}V_j(x)|\vartheta_{\delta_n,y_n}|^{2}dx>0.
\end{equation}
As a direct consequence of $(a)$ and $(b)$, we obtain that $\displaystyle\lim_{n\to \infty}\delta_{n}=\tilde{\delta}>0$.  By \eqref{e-4.18} and the assumption that $V_j(x)\in L^{2}(\R^4)$, we have
\begin{equation*}
  \lim_{n\to \infty} \|V_j\|_{L^{2}(B_{\delta_n}(y_n))}=0.
\end{equation*}
So, by H\"{o}lder inequality we get
\begin{equation*}
  \lim_{n\to \infty}\int_{\R^4}V_j(x)|\vartheta_{\delta_n,y_n}|^{2}dx\leq \lim_{n\to \infty} \Big[  \|V_j\|_{L^{2}(B_{\delta_n}(y_n))} \cdot \|\vartheta_{\delta_n,y_n}\|_{L^{4}(B_{\delta_n}(y_n))}^{2}\Big]=0,
\end{equation*}
which contradicts  \eqref{e-4.19}.
Thus, we have proved that also $(c)$ holds true.
\end{proof}

\begin{lem}\label{lm4.5}
Let $\langle x , y\rangle_{\R^4}$ denote the inner product of vector $x,y \in \R^4$ and let $r>0$ be fixed, then the following relations hold: \\
$(a)$~$\displaystyle\lim_{\delta\to 0}\sup \{ \gamma(\sqrt{k_1}\vartheta_{\delta,y}, \sqrt{k_2}\vartheta_{\delta,y}): y\in \R^4\}=0$;\\
$(b)$~$\displaystyle\lim_{\delta \to  \infty}\inf  \{ \gamma(\sqrt{k_1}\vartheta_{\delta,y}, \sqrt{k_2}\vartheta_{\delta,y}): |y|\leq r\}=1 $; \\
$(c)$~$\displaystyle \langle \xi( \sqrt{k_1}\vartheta_{\delta,y}, \sqrt{k_2}\vartheta_{\delta,y}), y  \rangle_{\R^4}>0$, $\forall y \in \R^4\setminus \{0\}$, $\forall \delta>0$.
\end{lem}
\begin{proof}
$(a)$ For every $\delta >0$ and $y\in \R^4$, we get
\begin{align*}
\left|\frac{y}{1+|y|}-\xi(\sqrt{k_1}\vartheta_{\delta,y},\sqrt{k_2}\vartheta_{\delta,y})\right|
 & \leq   \frac{\int_{\R^4}\Big| \frac{y}{1+|y|}-\frac{x}{1+|x|} \Big|(k_1+k_2) \vartheta_{\delta,y}^{4}dx}{(k_1+k_2)\int_{\R^4}\vartheta_{\delta,y}^{4}dx}\\
 &= \frac{\int_{B_{\delta}(y)}\Big| \frac{y}{1+|y|}-\frac{x}{1+|x|} \Big| \vartheta_{\delta,y}^{4}dx}{\int_{B_{\delta}(y)}\vartheta_{\delta,y}^{4}dx}\\
 & \leq 2\delta,
\end{align*}
where we have used the fact that
$
|\frac{y}{1+|y|}-\frac{x}{1+|x|}|<2\delta
$
for any $x\in B_{\delta}(y)$. Then
\begin{align*}
  0 &\leq \gamma(\sqrt{k_1}\vartheta_{\delta,y}, \sqrt{k_2}\vartheta_{\delta,y} ) \\
  &  = \frac{  \int_{B_{\delta_n}(y_n)}\Big| \frac{x}{1+|x|}-\xi(\sqrt{k_1}\vartheta_{\delta,y}, \sqrt{k_2}\vartheta_{\delta,y}) \Big| \vartheta_{\delta,y}^{4}dx}{\int_{B_{\delta_n}(y_n)}\vartheta_{\delta,y}^{4}dx} \\
  & \leq  \frac{  \int_{B_{\delta_n}(y_n)}\Big| \frac{x}{1+|x|}- \frac{y}{1+|y|} \Big| \vartheta_{\delta,y}^{4}dx}{\int_{B_{\delta_n}(y_n)}\vartheta_{\delta,y}^{4}dx} +
  \frac{  \int_{B_{\delta_n}(y_n)}\Big| \frac{y}{1+|y|}- \xi(\sqrt{k_1}\vartheta_{\delta,y}, \sqrt{k_2}\vartheta_{\delta,y})  \Big| \vartheta_{\delta,y}^{4}dx}{\int_{B_{\delta_n}(y_n)}\vartheta_{\delta,y}^{4}dx}\\
  &\leq 4\delta,
\end{align*}
which implies that
$$
\lim_{\delta \to 0} \gamma(\sqrt{k_1}\vartheta_{\delta,y}, \sqrt{k_2}\vartheta_{\delta,y} )=0,
$$
uniformly in $y\in \R^4$.

\smallskip

$(b)$ We first claim that
\begin{equation}\label{add-guo}
  \lim_{\delta \to \infty} \xi (\sqrt{k_1}\vartheta_{\delta,y}, \sqrt{k_2}\vartheta_{\delta,y})=0,
\end{equation}
uniformly for $y\in B_{r}(0)$. Indeed, since $\vartheta_{\delta,0}$ is radially symmetric, then
\begin{equation*}
  \int_{\R^4}\frac{x}{1+|x|}  \vartheta_{\delta,0}^{4}dx=0.
\end{equation*}
Therefore, by calculations we have
\begin{align*}
  |\xi(\sqrt{k_1}\vartheta_{\delta,y}, \sqrt{k_2}\vartheta_{\delta,y})|
  &=\frac{\Big|\int_{\R^4}\frac{x}{1+|x|}\vartheta_{\delta,y}^{4}dx \Big| }{\int_{\R^4}\vartheta_{\delta,y}^{4}dx  } = \frac{\Big|\int_{\R^4}\frac{x}{1+|x|}(\vartheta_{\delta,y}^{4}-  \vartheta_{\delta,0}^{4})dx \Big| }{\int_{\R^4}\vartheta_{\delta,y}^{4}dx  } \\
  & \leq \frac{1}{\mathcal{S}^2} \int_{\R^4} |\vartheta_{\delta,y}^{4}-  \vartheta_{\delta,0}^{4}|dx  = \frac{1}{\mathcal{S}^2} \int_{\R^4} |\vartheta_{1,\frac{y}{\delta}}^{4}-  \vartheta_{1,0}^{4}|dx \to 0,
\end{align*}
as $\delta \to \infty$, uniformly for $y\in B_{r}(0)$.

For any $\varepsilon>0$, we fix a constant $\rho=\rho(\varepsilon)>0$ such that $\frac{1}{1+\rho}< \frac{\varepsilon}{3}$.
For such a $\rho >0$, taking into account that
$
\lim_{\delta \to \infty} \int_{B_{\rho}(0)}\vartheta_{\delta,y}^{4}dx=0$,  {uniformly for} $y\in B_{r}(0)
$,
and  \eqref{add-guo}, we deduce that there exists $\delta_0>0$ such that
$$
\frac{1}{\mathcal{S}^2}\int_{B_{\rho}(0)}\vartheta_{\delta,y}^{4}dx<\frac{\varepsilon}{3}
$$
and
$$
|\xi(\sqrt{k_1}\vartheta_{\delta,y}, \sqrt{k_2}\vartheta_{\delta,y})|<\frac{\varepsilon}{3}
$$
for every $\delta \in (\delta_0, \infty)$ and $y\in B_{r}(0)$.
Observe that\footnote{Probably, $|\gamma()|<1$. Anyway, since this proof works, I do not compute anything more.}
\beq
\label{a1}
   \gamma(\sqrt{k_1}\vartheta_{\delta,y}, \sqrt{k_2}\vartheta_{\delta,y})
  \leq 1+ |\xi(\sqrt{k_1}\vartheta_{\delta,y}, \sqrt{k_2}\vartheta_{\delta,y})|
   \leq 1+ \frac{\varepsilon}{3}\qquad \forall \delta \in (\delta_0, \infty),\ \forall y\in B_{r}(0).
\eeq
  On the other hand,
\beq
\label{a2}
\begin{split}
  \gamma(\sqrt{k_1}\vartheta_{\delta,y}, \sqrt{k_2}\vartheta_{\delta,y})
  & = \frac{  \int_{\R^4}\Big| \frac{x}{1+|x|}-\xi(\sqrt{k_1}\vartheta_{\delta,y}, \sqrt{k_2}\vartheta_{\delta,y}) \Big| \vartheta_{\delta,y}^{4}dx}{\int_{\R^4}\vartheta_{\delta,y}^{4}dx} \\
 & \geq \frac{  \int_{\R^4} \frac{|x|}{1+|x|}  \vartheta_{\delta,y}^{4}dx}{\int_{\R^4}\vartheta_{\delta,y}^{4}dx} -  |\xi(\sqrt{k_1}\vartheta_{\delta,y}, \sqrt{k_2}\vartheta_{\delta,y})|            \\
  & \geq \frac{  \int_{\R^4\setminus B_{\rho}(0)} \frac{|x|}{1+|x|}  \vartheta_{\delta,y}^{4}dx}{\int_{\R^4}\vartheta_{\delta,y}^{4}dx} -  |\xi(\sqrt{k_1}\vartheta_{\delta,y}, \sqrt{k_2}\vartheta_{\delta,y})|\\
  &\geq \frac{\rho}{1+\rho}-\frac{\int_{B_{\rho}(0)}\vartheta_{\delta,y}^{4}}{\int_{\R^4}\vartheta_{\delta,y}^{4}dx}-|\xi(\sqrt{k_1}\vartheta_{\delta,y}, \sqrt{k_2}\vartheta_{\delta,y})|\\
  &\geq 1-\frac{1}{1+\rho}-\frac{\varepsilon}{3}-\frac{\varepsilon}{3}  \geq 1-\varepsilon\qquad \forall \delta \in (\delta_0, \infty),\ \forall y\in B_{r}(0).
\end{split}
\eeq
Then, by \eqref{a1} and \eqref{a2} we get $\lim_{\delta \to  \infty} \gamma(\sqrt{k_1}\vartheta_{\delta,z}, \sqrt{k_2}\vartheta_{\delta,z})=1$, uniformly for $y\in B_{r}(0)$.

\smallskip

$(c)$ Let $\delta>0$ and $y\in \R^4 \setminus \{0\}$. For any $x\in \R^4$ with $\langle x, y\rangle_{\R^4}>0$, there holds $|-x-y|>|x-y|$. Then by the properties of $\vartheta$, we deduce that $\vartheta_{\delta,y}(x)\geq \vartheta_{\delta,y}(-x)$ for any $x\in \R^4$ with $\langle x , y\rangle_{\R^4}>0$ and meas $\{x\in \R^4 \mid  \langle x, y\rangle_{\R^4}>0, \vartheta_{\delta,y}(x)>\vartheta_{\delta,y}(-x) \}>0$.
Thus, we have
\begin{align*}
\int_{\R^4}\frac{\langle x , y\rangle_{\R^4}}{1+|x|} \vartheta_{\delta,y}^{4}dx
  & =\int_{\{x\in \R^4 \mid \langle x , y\rangle_{\R^4}>0\}}\frac{\langle x , y\rangle_{\R^4}}{1+|x|} \vartheta_{\delta,y}^{4}dx+
  \int_{\{x\in \R^4 \mid \langle x , y\rangle_{\R^4}<0\}}\frac{\langle x , y\rangle_{\R^4}}{1+|x|} \vartheta_{\delta,y}^{4}dx \\
  & =\int_{\{x\in \R^4 \mid \langle x ,y \rangle_{\R^4}>0\}} \frac{\langle x , y\rangle_{\R^4}}{1+|x|} \big( \vartheta_{\delta,y}^{4}- \vartheta_{\delta,y}^{4}(-x) \big) dx\\
  &>0,
\end{align*}
which implies that $\langle \xi( \sqrt{k_1}\vartheta_{\delta,y}, \sqrt{k_2}\vartheta_{\delta,y}), y  \rangle_{\R^4}>0$ for  any $ y \in \R^4 \setminus \{0\}$ and $\delta>0$.
\end{proof}

\section{Proof of main results}\label{s5}

\qquad In this section, we will prove our main theorem, with the help of the previous results.

For each $\vartheta_{\delta,y}$, we  set
$$
\overline{\vartheta}_{\delta,y}:=( \sqrt{k_1}\vartheta_{\delta,y},  \sqrt{k_2}\vartheta_{\delta,y})
$$
and the projections
\begin{equation}
\label{pr}
 \widetilde{\vartheta}_{\delta,y}:= t_{\delta,y}\overline{\vartheta}_{\delta,y} \in \mathcal{N},
  \qquad
 \widehat{\vartheta}_{\delta,y}:= t_{\delta,y,0}\overline{\vartheta}_{\delta,y} \in \mathcal{N}_{0},
   \qquad
   t_{\delta,y}, \,\, t_{\delta,y,0}>0.
\end{equation}
Let us define
\begin{equation*}
  \xi\circ\overline {\vartheta}_{\delta,y}:=\xi( \sqrt{k_1}\vartheta_{\delta,y},  \sqrt{k_2}\vartheta_{\delta,y}), \qquad
  \gamma \circ \overline{\vartheta}_{\delta,y}:=\gamma( \sqrt{k_1}\vartheta_{\delta,y},  \sqrt{k_2}\vartheta_{\delta,y}),
\end{equation*}
and in analogous way
$\xi\circ\widetilde  {\vartheta}_{\delta,y}$,
$\xi\circ \widehat{\vartheta}_{\delta,y}$,
 $\gamma\circ\widetilde  {\vartheta}_{\delta,y}$,
$\gamma\circ \widehat{\vartheta}_{\delta,y}$.

Observe that, by \eqref{e-4.1},
\beq
\label{inv}
\xi\circ\overline {\vartheta}_{\delta,y}=\xi\circ\widetilde  {\vartheta}_{\delta,y}=\xi\circ \widehat{\vartheta}_{\delta,y},
\qquad
\gamma\circ\overline {\vartheta}_{\delta,y}=\gamma\circ\widetilde  {\vartheta}_{\delta,y}=\gamma\circ \widehat{\vartheta}_{\delta,y}
.\eeq

\begin{lem}\label{lm5.1}
Let $t_{\delta,y,0}$ be as in \eqref{pr}, then the  following relations hold: \\
$(a)$ \,$\displaystyle \lim_{\delta \to 0} \sup\{ |t_{\delta,y,0}-1|:y\in \R^4 \}=0$; \\
$(b)$ \,$\displaystyle \lim_{\delta \to \infty} \sup\{ |t_{\delta,y,0}-1|:y\in \R^4 \}=0$; \\
$(c)$ \,$\displaystyle \lim_{r \to  \infty} \sup\{ |t_{\delta,y,0}-1|:y\in \R^4, |y|=r,\delta>0 \}=0$.
\end{lem}
\begin{proof}
Note that $(\sqrt{k_1}\vartheta_{\delta,y}, \sqrt{k_2}\vartheta_{\delta,y})\in \mathcal{N}_{\infty}$, then by calculation  we get
\begin{align*}
1&=\frac{\int_{\R^4}\big(|\nabla (\sqrt{k_1}\vartheta_{\delta,y})|^{2}+|\nabla (\sqrt{k_2}\vartheta_{\delta,y})|^{2}\big)dx}{\int_{\R^4}\big[\mu_1(\sqrt{k_1}\vartheta_{\delta,y})^4+2\beta (\sqrt{k_1}\vartheta_{\delta,y})^2(\sqrt{k_2}\vartheta_{\delta,y})^2+\mu_2 (\sqrt{k_2}\vartheta_{\delta,y})^4  \big]dx} \\
&= \frac{t_{\delta,y,0}^{2}(k_1+k_2)\int_{\R^4}\vartheta_{\delta,y}^{4}dx-k_1\int_{\R^4}V_1(x)\vartheta_{\delta,y}^{2}dx
-k_2\int_{\R^4}V_2(x)\vartheta_{\delta,y}^{2}dx}{(k_1+k_2)\int_{\R^4}\vartheta_{\delta,y}^{4}dx}.
\end{align*}
By Lemma \ref{lm4.4}, we then obtain  $(a)$-$(c)$.
\end{proof}

\begin{lem}\label{lm5.2}
If $(A_1)$ and $(A_2)$ hold, then there exist constants $\bar{r}>0$ and $0<\delta_1<\frac{1}{2}<\delta_2$ such that
\begin{equation}\label{e-5.1}
\gamma \circ \widehat{\vartheta}_{\delta_1,y}<\frac{1}{2}, \quad \forall y\in \R^4,
\end{equation}
\begin{equation}\label{e-5.1-b}
\gamma \circ \widehat{\vartheta}_{\delta_2,y}>\frac{1}{2}, \quad \forall y\in \R^4, \ \ |y|<\bar{r}
\end{equation}
and
\begin{equation}\label{e-5.2}
  \sup\{I_{0} \circ\widehat{\vartheta}_{\delta,y} :(\delta,y) \in \partial \mathcal{H} \}<
   \bar{c} ,
\end{equation}
where $\bar{c}$ is defined in \eqref{e-4.14} and
\begin{equation*}
\mathcal{H}:=\{(\delta,y)\in \R^{+}\times \R^{4} : \delta\in [\delta_1,\delta_2], \,\, |y|<\bar{r}\}.
\end{equation*}
\end{lem}

\begin{proof}
First, observe that for each $\delta>0$ and $y\in \R^4$, we have
\begin{align*}
  I_0\circ \widehat{\vartheta}_{\delta,y}&=I_{0}(t_{\delta,y,0}\sqrt{k_1}\vartheta_{\delta,y},t_{\delta,y,0}\sqrt{k_2}\vartheta_{\delta,y}) \\
  &=\frac{(k_1+k_2)t_{\delta,y,0}^{2}}{2} \int_{\R^4}|\nabla \vartheta_{\delta,y}|^{2}dx+ \frac{t_{\delta,y,0}^{2}}{2} \int_{\R^4}\big( k_1V_{1}(x)+k_2V_2(x) \big)|\vartheta_{\delta,y}|^{2}dx \\
  &  \qquad -\frac{t_{\delta,y,0}^{4}}{4} \int_{\R^4}\big(\mu_1k_1^{2}+2\beta k_1k_2+\mu_2k_2^{2}\big)\vartheta_{\delta,y}^{4}  dx\\
  &=\frac{(k_1+k_2)t_{\delta,y,0}^{2}}{2} \int_{\R^4}|\nabla \vartheta_{\delta,y}|^{2}dx+   \frac{t_{\delta,y,0}^{2}}{2}  \int_{\R^4}\big( k_1V_{1}(x)+k_2V_2(x) \big)|\vartheta_{\delta,y}|^{2}dx \\
 &  \qquad -\frac{(k_1+k_2)t_{\delta,y,0}^{4}}{4}   \int_{\R^4}\vartheta_{\delta,y}^{4}  dx.
\end{align*}

Then, the existence of  $\delta_1\in (0,\frac{1}{2})$ such that $\gamma \circ \widehat{\vartheta}_{\delta_1,y}<\frac{1}{2}$  and   $I_{0}\circ \widehat{\vartheta}_{\delta_1,y}<\bar{c}$, for all $y\in \R^4$, follows from \eqref{inv}, Lemma \ref{lm4.5}$(a)$, Lemma \ref{lm4.4}$(a)$, Lemma \ref{lm5.1}$(a)$ and \eqref{e-4.15-1}.

Moreover, by Lemma \ref{lm4.4}$(c)$, Lemma \ref{lm5.1}$(c)$ and \eqref{e-4.15-1}, we can choose $\bar{r}>0$ such that, if $|y|=\bar{r}$, then $I_{0}\circ \widehat{\vartheta}_{\delta,y}<\bar{c}$ is satisfied for all $\delta>0$.

Finally, once $\bar{r}$ is fixed, by \eqref{inv}, Lemma \ref{lm4.5}$(b)$, Lemma \ref{lm4.4}$(b)$, Lemma \ref{lm5.1}$(b)$ and \eqref{e-4.15-1} again,  we obtain that there exists $\delta_2>0$ such that $\gamma \circ \widehat{\vartheta}_{\delta_2,y}>\frac{1}{2}$  and   $I_{0}\circ \widehat{\vartheta}_{\delta_2,y}<\bar{c}$ for all $y\in \R^4$ with $|y|\leq \bar{r}$.
\end{proof}

\begin{lem}\label{lm5.3}
 Let $\delta_1,\delta_2,\bar{r}$ and $\mathcal{H}$ be defined as Lemma \ref{lm5.2}. Then there exist $(\widetilde{\delta},\widetilde{y})\in \partial \mathcal{H}$ and
 $(\bar{\delta},\bar{y})\in  \mathring{\mathcal{H}}$ such that
 \begin{equation}\label{e-5.3}
\xi\circ \overline \vartheta_{\widetilde{\delta},\widetilde{y}}=0, \qquad \gamma\circ \overline \vartheta_{\widetilde{\delta},\widetilde{y}} \geq \frac{1}{2}
 \end{equation}
and
 \begin{equation}\label{e-5.4}
\xi\circ \overline \vartheta_{\bar{\delta},\bar{y}}=0,
\qquad
 \gamma\circ \overline \vartheta_{\bar{\delta},\bar{y}}= \frac{1}{2}.
 \end{equation}
\end{lem}

\begin{proof}
Obviously, \eqref{e-5.3} is trivial  if we choose $(\widetilde{\delta},\widetilde{y})=(\delta_2,0)$. Indeed, by the symmetric property of $\vartheta_{\delta_2,0}$, it is easy to see that $ \xi\circ \overline \vartheta_{\delta_2,0} =0$.
Moreover, by  Lemma \ref{lm5.2} together with property \eqref{inv}, we  obtain $\gamma\circ \overline \vartheta_{\delta_2,0 } \geq \frac{1}{2}$.
Therefore we have completed the first part of proof.

\smallskip

For any $(\delta,y)\in \mathcal{H}$, we define
\begin{equation*}
  \Theta(\delta,y)=\big(\gamma\circ \overline  \vartheta_{\delta,y},   \,  \xi\circ\overline  \vartheta_{\delta,y} \big)
\end{equation*}
and denote the homotopy map $\mathcal{T}: [0,1] \times \partial \mathcal{H} \to \R \times \R^N$ as following
\begin{equation}\label{e-5.5}
  \mathcal{T}(\delta,y,s)=(1-s)(\delta,y)+s \Theta(\delta,y).
\end{equation}
In order to prove \eqref{e-5.4}, it is enough to show that
\begin{equation}\label{e-5.6}
\deg\left(\Theta, \mathring{\mathcal{H}}, \left(\frac{1}{2},0\right)\right)=1.
\end{equation}
Note that $\deg\left(\Id, \mathring{\mathcal{H}}, \left(\frac{1}{2},0\right)\right)=1$.
If we prove that for each $(\delta,y)\in \partial \mathcal{H}$ and $s\in [0,1]$
\begin{equation}
\mathcal{T}(\delta,y,s)=  \big( (1-s)\delta+s\, \gamma\circ \overline  \vartheta_{\delta,y},
(1-s)y+s\, \xi\circ\overline  \vartheta_{\delta,y}  \big)\neq \left(\frac{1}{2},0\right),
\end{equation}
then  \eqref{e-5.6} holds true directly  due to the  topological degree homotopy invariance.
 Let $\partial \mathcal{H}=\mathcal{H}_{1}\cup \mathcal{H}_{2} \cup \mathcal{H}_{3}$ with
\begin{align*}
  \mathcal{H}_{1}&=\{(\delta,y)\in \partial \mathcal{H} :|y|\leq \bar{r}, \delta=\delta_1  \}, \\
  \mathcal{H}_{2}&=\{(\delta,y)\in \partial \mathcal{H} :|y|\leq \bar{r}, \delta=\delta_2  \}, \\
  \mathcal{H}_{3}&=\{(\delta,y)\in \partial \mathcal{H} :|y|=\bar{r}, \delta\in [\delta_1, \delta_2]  \}.
\end{align*}
If $(\delta,y)\in \mathcal{H}_{1}$, then by \eqref{inv} and \eqref{e-5.1}, we get
\begin{equation*}
(1-s)\delta_1+s\, \gamma\circ \overline  \vartheta_{\delta_1,y} <\frac{1}{2}(1-s)+\frac{s}{2}=\frac{1}{2}.
\end{equation*}
If $(\delta,y)\in \mathcal{H}_{2}$, then by \eqref{inv} and \eqref{e-5.1-b}, there holds
\begin{equation*}
(1-s)\delta_2+s\, \gamma\circ \overline  \vartheta_{\delta_2,y} >\frac{1}{2}(1-s)+\frac{s}{2} =\frac{1}{2}.
\end{equation*}
If $(\delta,y)\in \mathcal{H}_{3}$, then from Lemma \ref{lm4.5}$(c)$ we can deduce
\begin{equation*}
\langle (1-s)y+s\, \xi\circ\overline \vartheta_{\delta,y}  , y   \rangle_{\R^4}=(1-s)y^2+s\langle \xi(\sqrt{k_1}\vartheta_{\delta,y}, \sqrt{k_2}\vartheta_{\delta,y} ) , y\rangle_{\R^4} >0,
\end{equation*}
which implies
$$
(1-s)y+ +s\, \xi\circ\overline \vartheta_{\delta,y} \neq 0.
$$
\end{proof}

\begin{lem}\label{lm5.4}
Let $\delta_1$, $\delta_2$, $\bar{r}$ and $\mathcal{H}$ be defined as Lemma \ref{lm5.2}.
Assume that $(A_1)$-$(A_3)$ hold, then
\begin{equation}\label{e-5.8}
 \mathcal{K}=\sup\{I_{0}\circ \widehat{\vartheta}_{\delta,y}: (\delta,y ) \in \mathcal{H}\}< \min\left\{\frac{\mathcal{S}^{2}}{4\mu_1}, \frac{\mathcal{S}^{2}}{4\mu_2},2c_{\infty}\right\}.
\end{equation}
\end{lem}

\begin{proof}
Since $(\sqrt{k_1}\vartheta_{\delta,y}, \sqrt{k_2}\vartheta_{\delta,y} )\in \mathcal{N}_{\infty}$, it follows from Lemma \ref{lm2.2} that $t_{\delta,y,0}\geq 1$. By \eqref{e-4.15-1}, we have for every $(\delta,y)\in \mathcal{H}$,
\begin{align}\label{e-5.9}
  I_{0}\circ \widehat{\vartheta}_{\delta,y}&=\frac{t_{\delta,y,0}^{2}}{4}(k_1+k_2)\int_{\R^4} |\nabla\vartheta_{\delta,y}|^{2}dx+\frac{t_{\delta,y,0}^{2}}{4}\int_{\R^4}\big(k_1V_1(x)|\vartheta_{\delta,y}|^2+k_2V_{2}(x)|\vartheta_{\delta,y}|^2\big)dx \nonumber \\
  & \leq \frac{t_{\delta,y,0}^{2}}{4}(k_1+k_2)\|\vartheta_{\delta,y}\|_{D^{1,2}}^{2}+\frac{t_{\delta,y,0}^{2}k_{1}}{4}\|V_1\|_{L^{2}}\|\vartheta_{\delta,y}\|_{L^4}^{2}
  +\frac{t_{\delta,y,0}^{2}k_{2}}{4}\|V_2\|_{L^{2}}\|\vartheta_{\delta,y}\|_{L^4}^{2} \nonumber \\
  & \leq \frac{t_{\delta,y,0}^{2}}{4}(k_1+k_2)\|\vartheta_{\delta,y}\|_{D^{1,2}}^{2}+\frac{t_{\delta,y,0}^{2}k_1}{4\mathcal{S}}\|V_1\|_{L^{2}}\|\vartheta_{\delta,y}\|_{D^{1,2}}^{2}
  +\frac{t_{\delta,y,0}^{2}k_{2}}{4\mathcal{S}}\|V_2\|_{L^{2}}\|\vartheta_{\delta,y}\|_{D^{1,2}}^{2} \nonumber \\
  & = t_{\delta,y,0}^{2}\Big(1+\frac{k_1\|V_1\|_{L^2}}{(k_1+k_2)\mathcal{S}}+ \frac{k_2\|V_2\|_{L^2}}{(k_1+k_2)\mathcal{S}}     \Big)\Sigma.
\end{align}
Recalling that $(t_{\delta,y,0}\sqrt{k_1}\vartheta_{\delta,y},  t_{\delta,y,0}\sqrt{k_2}\vartheta_{\delta,y} )\in \mathcal{N}_{0}$, then
\begin{align}\label{e-5.10}
  t_{\delta,y,0}^{2}&= \frac{(k_1+k_2)\int_{\R^4}|\nabla \vartheta_{\delta,y}|^{2}dx+  \int_{\R^4}\big(k_1V_1(x)+k_2V_2(x)\big)|\vartheta_{\delta,y}|^{2}dx}{(k_1+k_2)\int_{\R^4}|\vartheta_{\delta,y}|^{4}dx} \nonumber \\
  & \leq 1+ \frac{k_1\|V_1\|_{L^{2}}+k_2\|V_2\|_{L^{2}}}{(k_1+k_2)\|\vartheta\|_{L^{4}}^{2}} \nonumber \\
  & \leq  1+\frac{k_1\|V_1\|_{L^2}}{(k_1+k_2)\mathcal{S}}+ \frac{k_2\|V_2\|_{L^2}}{(k_1+k_2)\mathcal{S}}.
\end{align}
Let us insert \eqref{e-5.10} into \eqref{e-5.9}, then by \eqref{e-4.15}, \eqref{e-4.14}, \eqref{e-4.15-1}, and taking into account \eqref{b2}, we get
\begin{align*}
I_{0}\circ \widehat{\vartheta}_{\delta,y}& \leq \Big(1+\frac{k_1\|V_1\|_{L^2}}{(k_1+k_2)\mathcal{S}}+ \frac{k_2\|V_2\|_{L^2}}{(k_1+k_2)\mathcal{S}}     \Big)^{2}\Sigma \\
& \leq  \Big(\mathcal{S}+\frac{k_1\|V_1\|_{L^2}}{(k_1+k_2)}+ \frac{k_2\|V_2\|_{L^2}}{(k_1+k_2)}     \Big)^{2} \frac{\bar{c}}{\mathcal{S}^2}  \\
&=  2^{-(1-a)} \min \Big \{ \frac{\beta^2-\mu_1\mu_2}{\mu_1(2\beta-\mu_1-\mu_2)}, \frac{\beta^2-\mu_1\mu_2}{\mu_2(2\beta-\mu_1-\mu_2)},2 \Big \}\bar{c}\\
& \leq  \min \Big\{\frac{\mathcal{S}^2}{4\mu_1}, \frac{\mathcal{S}^2}{4\mu_2}, 2c_{\infty}  \Big\}.
\end{align*}
\end{proof}

\begin{lem}\label{lm5.5}
 Let $\delta_1,\delta_2,\bar{r},\mathcal{H}$ be as in Lemma \ref{lm5.2}.
 If assumptions $(A_1)$ and $(A_2)$ are verified, then there exists a number $\lambda^*>0$ such that for each $\lambda \in (0,\lambda^*)$, the following relations hold:
 \begin{equation}\label{e-5.12}
   \gamma  \circ \widetilde{\vartheta}_{\delta_1,y} <\frac{1}{2}, \,  \, \forall y\in \R^4,
 \end{equation}
 \begin{equation}\label{e-5.12-b}
   \gamma \circ \widetilde{\vartheta}_{\delta_2,y}>\frac{1}{2},
   \ \  \forall y\in \R^4, \, \, |y|<\bar{r}
 \end{equation}
 and
 \begin{equation}\label{e-5.13}
   \tilde{\mathcal{K}}:=\sup\{I \circ \widetilde{\vartheta}_{\delta,y} :(\delta,y)\in \partial \mathcal{H}  \}<   \bar c  .
 \end{equation}
 Furthermore, if assumption $(A_3)$  also holds, then $\lambda^{**}$ can be found such that  for each $\lambda  \in (0, \lambda^{**})$, in addition to \eqref{e-5.12}-\eqref{e-5.13},
 \begin{equation}\label{e-5.14}
   \tilde{s}:=\sup\{I\circ \widetilde{\vartheta}_{\delta,y}:(\delta,y)\in  \mathcal{H} \}< \min \Big\{\frac{\mathcal{S}^2}{4\mu_1}, \frac{\mathcal{S}^2}{4\mu_2}, 2c_{\infty}  \Big\}
 \end{equation}
 is also satisfied.
 \end{lem}

 \begin{proof}
In view of \eqref{inv}, then inequalities \eqref{e-5.12} and \eqref{e-5.12-b} can be seen as a direct consequence of  \eqref{e-5.1} and \eqref{e-5.1-b} respectively.

Since   $\widetilde{\vartheta}_{\delta,y} \in \mathcal{N}$ and $   \widehat \vartheta_{\delta,y}  \in \mathcal{N}_{0}$  (see \eqref{pr}),  by  computation we get
 \begin{align*}
 t_{\delta,y}^{2}  &=\frac
 {(k_1+k_2)\int_{\R^4}|\nabla \vartheta_{\delta,y}|^{2}dx +k_1\int_{\R^4}(V_1(x)+\lambda )|\vartheta_{\delta,y}|^{2}dx+k_2\int_{\R^4}(V_2(x)+\lambda )|\vartheta_{\delta,y}|^{2}dx}
 {(\mu_1k_1^2+2\beta k_1k_2+ \mu_2 k_{2}^{2})\int_{\R^4}|\vartheta_{\delta,y}|^{4}dx} \\
   &= t_{\delta,y,0}^{2}+\frac{\lambda(k_1+k_2)\int_{\R^4} |\vartheta_{\delta,y}|^{2}dx }{(\mu_1k_1^2+2\beta k_1k_2+ \mu_2 k_{2}^{2})\int_{\R^4}|\vartheta_{\delta,y}|^{4}dx},
 \end{align*}
%
so, taking into account that
  \begin{equation}\label{add-g1}
    \int_{\R^N}\lambda|\vartheta_{\delta,y}|^{2}dx=\lambda\delta^{2}\int_{B_{1}(0)}|\vartheta|^{2}dx,
  \end{equation}
  we obtain
  \begin{equation}\label{add-g2}
 \lim_{\lambda \to 0}\sup_{(\delta,y)\in \mathcal{H}}|t_{\delta,y}-t_{\delta,y,0}|=0.
  \end{equation}
 Thus, if $\lambda$ is suitably small, then  \eqref{e-5.13} and \eqref{e-5.14} follow straightly by   \eqref{e-5.2}, \eqref{e-5.8}, \eqref{add-g1}
 and \eqref{add-g2}.
\end{proof}

In what follows,  we are ready to prove the existence of  solutions for system \eqref{S-4} by deformation arguments, that we can apply by the local  compactness result obtained in Section 3.
Before the proof, we recall the   notation for the sublevel sets:
\begin{equation*}
I^{c}:=\{ (u,v) \in \mathcal{N}: I(u,v)\leq c\},\qquad c\in\R,
\end{equation*}
and
\begin{equation*}
I_{0}^{c}:=\{ (u,v) \in \mathcal{N}_{0} : I_{0}(u,v)\leq c\},\qquad c\in\R.
\end{equation*}

\textbf{Proof of Theorem \ref{Th1.2}} \,\,  In the following, we always assume that  $\lambda\in (0,\lambda^*)$, where $\lambda^{*}$ is  fixed  in Lemma \ref{lm5.5}.
To prove Theorem \ref{Th1.2}, we will first   show that under assumptions $(A_1)$ and $(A_2)$ a critical level exists in  $(c_{\infty}, \bar{c})$.
 Then, we prove that another critical level  exists in $(\bar{c}, \min\{\frac{\mathcal{S}^{2}}{4\mu_1}, \frac{\mathcal{S}^{2}}{4\mu_2}, 2c_{\infty} \})$,  if in addition $(A_3)$  holds  and $\lambda\in(0,\lambda^{**})$, where $\lambda^{**}$ is fixed in Lemma \ref{lm5.5}.

\smallskip

Let us assume that  $(A_1)$ and $(A_2)$ hold.
By using \eqref{e-4.9}, \eqref{e-5.3}, \eqref{inv},  \eqref{e-5.13},  \eqref{e-4.14}  we have
 \begin{equation}\label{e-4.24}
   c_{\infty}< c^{**}\leq  I \circ \widetilde{\vartheta}_{\tilde{\delta},\tilde{y}} \leq
   \tilde{\mathcal{K}}< \bar{c}< c^{*}_0.
 \end{equation}
We claim that a critical level of $I$ constrained on $\mathcal{N}$ exists in  $[c^{**}, \tilde{\mathcal{K}}]$.
Arguing by contradiction, we  suppose that this is not true.
From \eqref{e-4.24}, \eqref{e-4.14}, and Corollary \ref{co3.4},  there follows that the  functional $I$ constrained on $\mathcal{N}$ satisfies the Palais-Smale condition in the energy range  $[c^{**}, \tilde{\mathcal{K}}]$.
Then, by  deformation lemma (see Lemma 2.3 \cite{WM}), we can find a positive constant $\sigma_1$ such that $c^{**}-\sigma_1>c_{\infty}$, $\tilde{\mathcal{K}}+\sigma_1<\bar{c}$, and a continuous function
$\eta:[0,1]\times I^{\tilde{\mathcal{K}}+\sigma_1}\to I^{\tilde{\mathcal{K}}+\sigma_1}
$
such that
\begin{equation}\label{e-5.16}
\begin{aligned}
  \eta  \big(0,(u,v)\big)&=(u,v) \ \ \ \ \ \ \ \ \forall (u,v)\in I^{\tilde{\mathcal{K}}+\sigma_1}; \\
  \eta  \big(s,(u,v)\big) &=(u,v) \ \ \ \ \ \ \ \ \forall (u,v) \in I^{c^{**}-\sigma_1}, \,\, \forall s\in[0,1];  \\
  I\circ \eta  \big(s,(u,v)\big)&\leq I(u,v) \ \ \ \ \ \ \  \forall (u,v)\in I^{\tilde{\mathcal{K}}+\sigma_1}, \,\, \forall s\in [0,1]; \\
  \eta  (1,I^{\tilde{\mathcal{K}}+\sigma_1})&\subset I^{c^{**}-\sigma_1}.
\end{aligned}
\end{equation}
Therefore,  by \eqref{e-5.13} and \eqref{e-5.16}, we get
 \begin{equation}\label{e-5.18}
   (\delta,y)\in \partial \mathcal{H} \Longrightarrow   I \circ  \widetilde{\vartheta}_{\delta,y}  \leq \tilde{\mathcal{K}}\Longrightarrow   I \circ \eta  (1,\widetilde{\vartheta}_{\delta,y})  \leq c^{**}-\sigma_1.
 \end{equation}
Let $s\in [0,1]$ and $(\delta,y)\in \partial \mathcal{H}$, we define a map
 \begin{equation*}
   \Gamma(\delta,y,s)=
   \begin{cases}
    \mathcal{T}(\delta,y,2s), \ \  &s\in [0,\frac{1}{2}]; \\
    \big( \gamma\circ \eta (2s-1, \widetilde{\vartheta}_{\delta,y}), \,\, \xi\circ \eta(2s-1, \widetilde{\vartheta}_{\delta,y})  \big), \ \ & s\in [\frac{1}{2},1 ],
   \end{cases}
 \end{equation*}
  where $ \mathcal{T}$ is the map defined in \eqref{e-5.5}.
  On the one hand,   we have already proved in Lemma \ref{lm5.3} that
  \begin{equation}\label{e-5.19}
    \Gamma(\delta,y,s)\neq \left(\frac{1}{2},0\right)\qquad  \forall s\in \left[0,\frac{1}{2}\right], \ \ \forall (\delta,y)\in \partial \mathcal{H}.
  \end{equation}
 On the other hand, in virtue of \eqref{e-5.16}, \eqref{e-5.13}, \eqref{e-4.14} and \eqref{c*},  we obtain
  \begin{equation*}
  I \circ \eta  (2s-1,\widetilde{\vartheta}_{\delta,y})\leq I \circ  \widetilde{\vartheta}_{\delta,y} \leq \tilde{\mathcal{K}} <\bar{c}< c^{*}_0\leq c^*,
    \ \  \forall s\in\left[\frac{1}{2},1\right], \,\, (\delta,y)\in \partial\mathcal{H},
  \end{equation*}
which leads to
\begin{equation}\label{e-5.20}
  \ \ \Gamma(\delta,y,s)\neq \left(\frac{1}{2},0\right)\qquad \forall s\in \left[\frac{1}{2},1\right], \ \ \forall (\delta,y)\in \partial \mathcal{H}.
\end{equation}
By \eqref{e-5.19}, \eqref{e-5.20} and the continuity of $\Gamma$, we see that there exists $(\check{\delta}, \check{y}) \in \partial \mathcal{H}$ such that
\begin{equation}\label{e-5.21}
  \xi\circ \eta  (1,\widetilde{\vartheta}_{\check{\delta}, \check{y}})=0, \ \ \ \ \gamma \circ \eta  (1, \widetilde{\vartheta}_{\check{\delta},\check{y}})\geq \frac{1}{2}.
\end{equation}
Then,  by using the definition of $c^{**}$ and \eqref{e-5.21}, we  have
$$
I\circ \eta  (1,\widetilde{\vartheta}_{\check{\delta},\check{y}})\geq c^{**},
$$
which contradicts with \eqref{e-5.18}.
Thus, we succeed in proving that functional $I$ constrained on $\mathcal{N}$ has at least a  critical point $(u_{l}, v_{l})\in \mathcal{N}$ such that $I(u_{l}, v_{l})\in (c_{\infty},\bar{c})$.
As a consequence, $(u_{l}, v_{l})$ is a critical  point of $I$, by standard arguments, and
moreover we get that  $(u_{l}, v_{l})$ is a nontrivial  non-negative solution of system \eqref{S} by Remark \ref{rem+} and  Corollary  \ref{nontrivial}.

\smallskip

Now, we suppose that $(A_3)$ holds, too.
Then  by using  \eqref{e-4.14}, \eqref{c*}, \eqref{e-5.4}, \eqref{inv}  and \eqref{e-5.14}, we get
\begin{equation}
\label{eC}
 c_\infty< \bar{c}<c^*_0< c^* \leq I \circ  \widetilde{\vartheta}_{\bar{\delta},\bar{y}} \leq \tilde{s}<\min \Big\{\frac{\mathcal{S}^2}{4\mu_1}, \frac{\mathcal{S}^2}{4\mu_2}, 2c_{\infty}  \Big\}.
\end{equation}
We claim that the functional $I$ constrained on $\mathcal{N}$ has a critical level in the interval $[c^*, \tilde s]$.
Arguing as above, we assume by contradiction that there is no critical level  in $[c^*, \tilde s] \! \subset
(\bar{c}, \min \big\{\frac{\mathcal{S}^2}{4\mu_1}, \frac{\mathcal{S}^2}{4\mu_2}, 2c_{\infty}  \big\})$.
By \eqref{eC} and Corollary \ref{co3.4}, we can apply the deformation lemma  again.
So, there exist a positive number $\sigma_2$ such that
$$c^*-\sigma_2>\bar{c}, \quad \tilde{s}+\sigma_2< \min \left\{\frac{\mathcal{S}^2}{4\mu_1}, \frac{\mathcal{S}^2}{4\mu_2}, 2c_{\infty} \right\}$$
 and a continuous function $\tilde{\eta}:[0,1]\times I^{\tilde{s}+\sigma_2}\to I^{\tilde{s}+\sigma_2}$
such that
\begin{equation*}
\begin{aligned}
  \tilde{\eta}  \big(0,(u,v)\big)&=(u,v) \ \ \ \ \ \ \ \ \forall (u,v)\in I^{\tilde{s}+\sigma_2}; \\
  \tilde{\eta}   \big(s,(u,v)\big) &=(u,v) \ \ \ \ \ \ \ \ \forall (u,v) \in I^{c^{*}-\sigma_2}, \,\, \forall s\in[0,1];  \\
  I\circ \tilde{\eta} \big(s,(u,v)\big)&\leq I(u,v) \ \ \ \ \ \ \  \forall (u,v)\in I^{\tilde{s}+\sigma_2}, \,\, \forall s\in [0,1]; \\
  \tilde{\eta}  (1,I^{\tilde{s}+\sigma_2})&\subset I^{c^{*}-\sigma_2}.
\end{aligned}
\end{equation*}
We remark that in particular
\begin{equation}\label{g}
 \forall (\delta,y)\in \mathcal{H}\Rightarrow   I\circ \widetilde{\vartheta}_{\delta,y}  \leq \tilde{s}\Rightarrow   I\circ \tilde{\eta}  (1,\widetilde{\vartheta}_{\delta,y} )  \leq c^*-\sigma_2.
\end{equation}
For any $s\in [0,1]$ and $(\delta,y)\in \mathcal{H}$, we define the map
\begin{equation*}
   \tilde{\Gamma}(\delta,y,s)=
   \begin{cases}
    \mathcal{T}(\delta,y,2s), \ \  &s\in [0,\frac{1}{2}]; \\
    \big( \gamma\circ \tilde{\eta} (2s-1, \widetilde{\vartheta}_{\delta,y}), \,\, \xi\circ \tilde{\eta}(2s-1, \widetilde{\vartheta}_{\delta,y})  \big), \ \ & s\in [\frac{1}{2},1 ],
   \end{cases}
 \end{equation*}
  where  the map $ \mathcal{T}$ is defined in \eqref{e-5.5}.
  As proved in Lemma \ref{lm5.3},
  $$
   \tilde{\Gamma}(\delta,y,s)\neq \left(\frac{1}{2},0\right)\qquad \forall s\in \left[0,\frac{1}{2}\right], \ \ \forall (\delta,y)\in \partial \mathcal{H}.
  $$
Moreover, taking into account
$$
(\delta,y)\in \partial \mathcal{H}\Rightarrow   I \circ \widetilde{\vartheta}_{\delta,y} \leq \tilde{\mathcal{K}} <\bar{c}<c^{*}-\sigma_2,
$$
 we have
$$
\tilde{\Gamma}(\delta,y,s)= \tilde{\Gamma}\left(\delta,y,\frac{1}{2}\right)=\mathcal{T}(\delta,y,1), \quad   \forall s\in \left[\frac{1}{2},1\right], \ \forall (\delta,y)\in \partial \mathcal{H}.
$$
Then,
 $$
   \tilde{\Gamma}(\delta,y,s)\neq \left(\frac{1}{2},0\right)\qquad \forall s\in \left[\frac{1}{2},1\right], \ \ \forall (\delta,y)\in \partial \mathcal{H}.
  $$
Hence, a pair $(\hat{\sigma},\hat{{{y}}})\in \mathcal{H}$ must exist  such that
$$
\xi\circ \tilde{\eta}(1,\widetilde{\vartheta}_{\hat{\delta},\hat{y}})=0, \ \ \gamma \circ \tilde{\eta}(1,\widetilde{\vartheta}_{\hat{\delta},\hat{y}})=\frac{1}{2}.
$$
Then we get
$$I\circ \tilde{\eta}  (1,\widetilde{\vartheta}_{\hat{\delta},\hat{y}})\geq c^*, $$
which contradicts \eqref{g}. Therefore the constrained functional $I|_{\mathcal{N}}$ has at least a critical point $(u_h,v_h)$ with  $I(u_h,v_h) \in (\bar{c}, \min\{\frac{\mathcal{S}^2}{4\mu_1},\frac{\mathcal{S}^2}{4\mu_2}, 2c_{\infty}\})$.
Arguing as for the low energy critical point $(u_l,v_l)$, it turns out that $(u_h,v_h)$ is also a nontrivial  non-negative   solution of system \eqref{S}.
{\hfill$\blacksquare$\vspace{6pt}}


The proof of Proposition \ref{Preg} is standard, so we only outline it.
 For the classical regularity results, we refer the reader for example to  \cite[\S 10]{LLbook}, and in particular Theorem  10.2  therein, or to \cite[\S 8]{GiTr01book}.

 {{\ni \bf \underline{Sketch  of the proof of Proposition \ref{Preg}:} }} $(a)$ Suppose that $V_j\in L^{q_j}_{\lloc}(\R^4)$ with $q_j>2$, $j=1,2$.
 Let us observe that the function $u$ satisfy
  $$
  -\Delta u=a(x)u,
  $$
  where
 $$
 a(x)=-(V_1(x)+\lambda)+\mu_1u^{2}+\beta v^{2}\in L^2_{\lloc}(\R^4),
 $$
 so $u\in L^q_{\lloc}(\R^4)$ for every $q\in[1,\infty)$ (see, for example, \cite[Lemma B.3]{Sbook}).
The same argument shows $v\in L^q_{\lloc}(\R^4)$ for every $q\in[1,\infty)$.
Then  we can write
\beq
\label{ereg}
  -\Delta u=b_1(x),\qquad -\Delta v=b_2(x)
\eeq
with $b_j\in L^{\bar q_j}_{\lloc}(\R^4)$, for $\bar q_j\in(2,q_j)$, $j=1,2$,  and hence $u,v\in  \mathcal{C}^{0,\alpha}(\R^4)$   by classical regularity results.
If $u\ge 0$, $v\ge 0$ and $u\not\equiv 0$, $v\not \equiv 0$, then $u(x),v(x)>0$, $\forall x\in\R^4$ by the Harnak inequality (see \cite[Theorem 7.2.1]{PSbook})

\smallskip

$(b)$  If we assume $V_j\in L^{q_j}_{\lloc}(\R^4)$ with $q_j>  4$, and argue as in the previous point, then we obtain \eqref{ereg} with $b_j\in L^{\bar q_j}_{\lloc}(\R^4)$, for $\bar q_j\in(4,q_j)$.
Then we are in position to apply classical regularity results  and so  the desired conclusion holds.

\smallskip

$(c)$  If we assume $V_j\in\cC^{0,\alpha_j}_{\lloc}(\R^4)$, then in \eqref{ereg} we have $b_j\in \cC^{0,\alpha_j}_{\lloc}(\R^4)$, which implies $u,v\in \cC^2(\R^4)$ (see  f.i.  \cite[Theorem 10.3]{LLbook}).

{\hfill$\blacksquare$\vspace{6pt}}

\vspace{.55cm}
\noindent {\bf Acknowledgements:}
Lun Guo was supported by the National Natural Science Foundation of China (No.11901222) and the China Postdoctoral Science Foundation (No.2021M690039).
Qi Li was supported by the National Natural Science Foundation of China (No. 12201474).
Xiao Luo was supported by the National Natural Science Foundation of China (No.11901147).
Riccardo Molle has been supported by the INdAM-GNAMPA group;
he acknowledges also the MIUR Excellence Department Project awarded to the Department of Mathematics, University of Rome ``Tor Vergata'', CUP E83C18000100006.

\end{document}